\documentclass[11pt,a4paper]{amsart}

\usepackage{graphicx}
\usepackage{enumerate}
\usepackage{amsmath, amssymb, amsthm, amscd}
\usepackage{pst-all}
\usepackage{url}
\usepackage{caption}
\usepackage{xcolor}
\usepackage{hyperref}

\theoremstyle{plain}
\newtheorem{theorem}{Theorem}[section]

\newtheorem{proposition}[theorem]{Proposition}

\theoremstyle{definition}
\newtheorem{definition}[theorem]{Definition}

\theoremstyle{remark}
\newtheorem{remark}[theorem]{Remark}

\numberwithin{equation}{section}

\newcommand{\D}{\partial}

\newcommand{\Dlt}{\Delta t}

\newcommand{\Dt}{\partial_t}
\newcommand{\dvg}{\nabla \cdot }
\newcommand{\Dxm}{\partial_{x_m}}

\newcommand{\grd}{\nabla}

\newcommand{\ld}{\lambda}

\newcommand{\mbb}{\mathbb}

\newcommand{\mcal}{\mathcal}

\newcommand{\Norm}[1]{{\left\vert\kern-0.25ex\left\vert\kern-0.25ex\left\vert #1 
    \right\vert\kern-0.25ex\right\vert\kern-0.25ex\right\vert}}

\newcommand{\rf}{\mathrm{ref}}

\newcommand{\veps}{\varepsilon}

\makeindex             

\begin{document}

\title[AP SI-IMEX Schemes for the EP System]{High Order Asymptotic
  Preserving and Classical Semi-implicit RK Schemes for the
  Euler-Poisson System in the Quasineutral Limit}        

\author[Arun]{K. R. Arun}
\thanks{K.~R.~A.\ acknowledges Core Research Grant - CRG/2021/004078
  from Science and Engineering Research Board, Department of Science
  \& Technology, Government of India.} 
\address{School of Mathematics, Indian Institute of Science Education
  and Research Thiruvananthapuram, Thiruvananthapuram - 695551,
  India.} 
\email{arun@iisertvm.ac.in}

\author[Crouseilles]{N. Crouseilles}
\address{Univ Rennes, Inria Bretagne Atlantique (MINGuS) \& IRMAR UMR
  6625 \& ENS Rennes, France}  
\email{nicolas.crouseilles@inria.fr}

\author[Samantaray]{S. Samantaray}
\address{Centre for Applicable
  Mathematics, Tata Institute of Fundamental Research, Bangalore - 560065, India.} 
\email{saurav22@tifrbng.res.in}

\date{\today}



\begin{abstract}
 In this paper, the design and analysis of high order accurate IMEX
 finite volume schemes for the compressible Euler-Poisson (EP)
 equations in the quasineutral limit is presented. As the quasineutral
 limit is singular for the governing equations, the time discretisation is
 tantamount to achieving an accurate numerical method. To this end,
 the EP system is viewed as a differential algebraic equation system
 (DAEs) via the method of lines. As a consequence of this vantage
 point, high order linearly semi-implicit (SI) time discretisation are
 realised by employing a novel combination of the direct approach used
 for implicit discretisation of DAEs and, two different classes of
 IMEX-RK schemes: the additive and the multiplicative. For
 both the time discretisation strategies, in order to account for
 rapid plasma oscillations in quasineutral regimes, the nonlinear
 Euler fluxes are split into two different combinations of stiff and 
 non-stiff components. The high order scheme resulting from the
 additive approach is designated as a classical scheme while the one
 generated by the multiplicative approach possesses the asymptotic
 preserving (AP) property. Time discretisations for the classical and
 the AP schemes are performed by standard IMEX-RK and SI-IMEX-RK
methods, respectively so that the stiff terms are treated implicitly
and the non-stiff ones explicitly. In order to discretise in space a
Rusanov-type central flux is used for the non-stiff part, and simple
central differencing for the stiff part. AP property is also established
for the space-time fully-discrete scheme obtained using the
multiplicative approach. Results of numerical experiments are
presented, which confirm that the high order schemes based on the
SI-IMEX-RK time discretisation achieve uniform second order
convergence with respect to the Debye length and are AP in the
quasineutral limit.
\end{abstract}

\maketitle
\section{Introduction}
\label{sec:Intro}
Physics classifies the matter in the universe broadly into four
different states: solids, liquids, gases and plasma. Even though the
fourth state is what constitutes almost the entire universe, its found
in rarity in our proximity. At the same time matter in the form of
plasma has far reaching scientific and industrial
relevance. Mathematically, a plasma is primarily studied based on
two different kinds of models: kinetic models and fluid models. Both
these fields are indeed very active areas of research and with time
new frontiers are being explored and established within them. Fluid
models have the density, the temperature and the mean velocity as
the base components under observation. Whereas kinetic models
incorporate the time evolution of a distribution function which gives
the probability of particles being in a given state in the
six-dimensional phase space at a given time \cite{BL91}. There are
undoubtedly situations where kinetic models are inevitable for the
application under consideration. But, at the same time there are many
situations in which fluid models are adequate enough. Owing to their
simplicity in dimensional requirements, fluids models are much easier
to numerically approximate and simulate. We refer the reader to
\cite{Che13, GKP19, HM13, KT86} for more details on the physics of
plasma, their mathematical models and analysis.  

In this paper we consider the one-fluid Euler-Poisson (EP) system
\cite{Deg13, DLV08, Fab92} which is a macroscopic hydrodynamic
model that simulates the dynamics of electrons in a plasma with
constant positive ion background. It is a rather simple fluid model,
albeit comprising of most of the numerical pathologies offered by more
complex models. Therefore, it serves as a perfect test bed for the
design and analysis of a new class of higher order methods. The
one-fluid EP system of equations, is in itself a simplification of a
more physically relevant plasma model, the two-fluid EP plasma model
\cite{Che13, Jue01}. The two-fluid EP system consists of equations for
the conservation of masses and momentum for each of electrons and
protons, which are coupled with a Poisson equation for the
electrostatic potential. This  constitutes a hyperbolic elliptic
coupling. The literature on the analysis of the EP system leading to
the existence and uniqueness of solution is quite vast; we refer the
interested reader to
\cite{Bez93,BK18,Gam93,Mak86,Mak92,Mak93,MP90,MU87,MN95} and the 
references cited therein, for an exposure of the same. For a detailed
analytical study of the EP system, the reader is referred to, e.g.\
\cite{CP98, PRV95}. The two fluid model gives rise to two important
physical scales: the Debye length and the electron plasma period; see 
\cite{Che13, KT86} for details. The Debye length $\ld_D$ is a measure
of the typical length scale of charge imbalances in plasma, and the
electron plasma period $\tau_p$ is the period of oscillations which
occur (due to electrostatic restoring force) when such imbalances of
charges take place. In the quasineutral region, the electric charges
vanish and simultaneously, the electron plasma period reduces
drastically. As a consequence, high frequency plasma oscillations are
generated. Classical  discretisation procedures fail to capture these
oscillations unless their space and time mesh sizes resolve the Debye
length and the electron plasma period. In other words, the space and
time steps, $\Delta x$ and $\Dlt$, have to be severely restricted by
\begin{equation}
  \label{eq:dxdt_qn}
  \Delta x \leq \lambda_D, \quad \Dlt \leq \tau_p.
\end{equation}
Explicit discretisation methods where the discretisation parameters
are constrained as in \eqref{eq:dxdt_qn} are naturally computationally
expensive and are more prone to developing instabilities. 

There is an abundance of literature on methods which attempt to
overcome the mesh size restrictions imposed by the presence of
quasineutral regimes in the flow. These methods are developed for
either kinetic models or fluid models. The literature in the case of
kinetic models is quite vast; see \cite{BF82, CLF_JCP_82, LCF_JCP_83,
  Mason_JCP_81, Mason_JCP_83, Mason_MTS_85, Mason_JCP_87, RD_JCP_91,
  RD_JCP_92, WBF_JCP_86} and the references therein. However, the
literature on fluid models is less plentifully rich. For some of the
noted developments focusing on fluid models, we refer to
\cite{CDW_JCP_99, Fab92, HN78, SM_IJNM_05, SL_JCP_03}. However, in
many practical problems of high relevance, one encounters the
co-existence of quasineutral and non-quasineutral regions within the
same problem. Therefore flow characteristics are multiscaled in nature.
When subjected to appropriate scalings the multiscale features of
plasma flows is often characterised by one or more singular
perturbation parameters in the governing equations, in a
non-dimensional form  \cite{CG00, DM08, GLM+01, GHR13, HL09, JLL09,
  JV07, Wan04}.The quasineutral limit is a singular limit of the EP
system. Owing to this multiscaled nature, numerical  
simulation of plasmas is still an eminent challenge for the scientific
community. It plagues numerical methods not only with outrageous
stability restrictions but also with inaccuracies in the asymptotic
regimes, especially close to the singular limit. This pathology yearns
for the need to address the numerical approximations of plasma models
in the asymptotic preserving (AP) framework \cite{Deg13, jin-ap,
  Jin12}. An AP scheme has two main characteristic features, exactly
aiming to remedy the above noted challenges associated with the 
numerical approximation of plasma flows, which are:
\begin{itemize}
\item its stability requirements are independent of the multiscale or
  a singular parameter;
\item in the limit of the singular parameter, the numerical scheme
  transforms itself to a scheme for the limit system.
\end{itemize}
It has been reported in the literature that the ability of a scheme to
be AP gets primarily dictated by the choice of the time integrator,
rather than the spatial approximation techniques. That being said, it
doesn't mean that caution shouldn't be exercised while choosing the
spatial discretisation. To this end, semi-implicit Runge-Kutta(RK) or 
Implicit Explicit (IMEX)-RK \cite{PR01, PR05} stiff ordinary differential
equation (ODE) integrators have played a dominating role in dispensing
the tasking duty of tackling the stiffness in the system, and the
limiting singular nature of the governing equations in hand. Recently
a new class of IMEX-RK schemes called semi-implicit (SI)-IMEX-RK
schemes for stiff ODEs has been developed in \cite{BFR16}. The
difference between the new class i.e. the SI-IMEX-RK schemes and the
earlier (additively partitioned)  IMEX-RK schemes \cite{PR01, PR05} is
that in case of the SI-IMEX-RK schemes there is no need to identify an
additive stiff and non-stiff part in the ODE. The stiff variable can be
determined implicitly. This gives a massive advantage in systems where
the stiff variables cannot be optimally split into summable
stiff-non-stiff parts. Moreover, in many cases they aid in obtaining
methods which are only linearly implicit.  

AP schemes based on IMEX time discretisation methods have found
widespread relevance in a wide array of problems with singular
parameters, such as: the low Mach  number Euler equation \cite{AS20, 
  BispArunLucaNoe, BMY17, BQR+19, CDK12, DT11, DLV17, 
  HJL12, Kle95, KBS+01, MRK+03, NBA+14, SBG+99, Tan12}, to name 
but a few. First order IMEX schemes for plasma flow problems have been
developed and studied in the literature, we refer the reader to
\cite{CDV_ASP_05, CDV07, Deg13, DLV08, Neg13} for a few developments
in this direction. However, the methods proposed in these references
are all first order accurate in time and space. In this paper we are
interested to develop a general platform to design high-order IMEX
schemes for plasma dynamics problems using the one-fluid EP model as a
prototype system. In \cite{DP13} authors have developed higher order
semi-implicit schemes based on additively partitioned IMEX-RK schemes
for hyperbolic kinetic models of plasma flows. Recently, SI-IMEX-RK
based high order schemes for hyperbolic-elliptic kinetic models of
plasma have been developed and analysed in \cite{FR16, FR17,FRZ21}. In
this paper our endeavour is to lay out a path way which facilitates
the design of high order semi-implicit schemes for hyperbolic-elliptic
coupled systems. To this end, via method of lines we perceive the coupled
system as an index-1 differential algebraic system (DAEs). The time
discretisation is designed by combining the direct approach
\cite{HW96} to design implicit RK schemes for DAEs with both the
additively partitioned IMEX-RK schemes and the SI-IMEX
schemes. Therefore this paper has also an underlying focus beyond the 
development of high-order schemes for the EP system i.e. this also
showcases an approach to design high-order semi-implicit schemes for
stiff DAEs as was done in \cite{PR01, PR05} for stiff ODEs, which is a
novelty in itself.

The two main challenges in getting higher order accuracy are the
following. First, the right hand side of the momentum equation
contains a term which is non-linear and only one of the components in
this non-linear product survives in the quasineutral limit. So, any AP
discretisation should treat that component implicitly. In order to
reduce the implicitness, the other component in the non-linear product
should be explicit. In other words, we have a product of explicit and
implicit terms, necessitating the use of SI-IMEX schemes instead of
IMEX. Second, the Poisson equation has to be enforced at every stage 
of the RK time stepping. Since the Poisson equation is time
independent, it has to be treated as an algebraic relation as far as
the time discretisation is concerned. Formal asymptotic analysis of
the schemes developed on both the additively partitioned IMEX-RK
and the SI-IMEX-RK platforms reveals that the former is not
asymptotically consistent, but the later is. This is in agreement with
the analysis presented in \cite{Deg13} regarding the first-order 
classical and AP schemes, of which the schemes designed in this paper
turn out to be higher order extensions.  

The rest of this paper is organised as follows. In
Section~\ref{ch:qn_sec2}, we present the governing equations which
consider the movement of the electrons within a constant non evolving
positively charged ion background, leading to the so-called one-fluid
EP system \cite{Deg13, DLV08}. Furthermore results
pertaining to asymptotic analysis of the one-fluid model from the
literature \cite{Deg13}, which aid in the analysis of the
numerical schemes, are also presented in this section. Particularly,
equivalent reformulations of the EP system which aid in obtaining the  
AP property of the schemes are furnished. In
Section~\ref{sec:ch_qn_time_sd}, a linearised version of the one
dimensional EP system is derived. Upon Fourier transformation in space
the linearised system turns out to be an index-1 DAEs
\cite{HW96}. Motivated by this the EP system is cast into a system of
index-1 DAEs via method of lines. A brief discussion on the direct
approach to design implicit RK schemes for index-1 DAEs, additively
partition IMEX-RK schemes \cite{PR01, PR05} and SI-IMEX-RK schemes
\cite{BFR16}, is also included in this section. In
Section~\ref{sec:ch_qn_time_semi-disc} we combine the direct
approach for index-1 DAEs with the additively partitioned IMEX-RK
schemes and SI-IMEX-RK schemes to obtain two classes of high-order
semi-implicit schemes for the index-1 DAEs, in in-turn for the EP
system. Section~\ref{sec:anal_TSD} is devoted to the asymptotic
analysis of the developed high-order schemes. Therein the additively
partitioned IMEX-RK schemes are shown to be not consistent with the
asymptotic limit system, whereas the SI-IMEX-RK scheme is shown to be
AP being consistent with the incompressible limit system. In
Section~\ref{sec:ep_space-time_fully-disc}, the space discretisation
technique is presented based on the finite volume framework in which a
Godunov-type scheme for the explicit fluxes and a simple central
differencing for the implicit gradient terms are used. The fully
discrete scheme is also proved to be asymptotically consistent. In
Section~\ref{sec:ep_numer_case_stud} some numerical case studies are
carried out to test and display the high-order accuracy, AP property
and the superior performance of the SI-IMEX-RK schemes to the IMEX-RK
schemes. Finally, the paper is closed in Section~\ref{sec:ch6_con}
with some conclusions and future plans for further extension of the
schemes to hyperbolic elliptic systems.  

\section{One-fluid Euler-Poisson System and its Quasineutral Limit}  
\label{ch:qn_sec2}
The one-fluid EP system is a hydrodynamic model of plasma electrons
under the action of an electrostatic force. The governing equations of
this model are the compressible Euler equations for the conservation
of mass and momentum, coupled with a Poisson equation for the electric
potential. The scaled non-dimensionalised system of equations reads:
\begin{align}
  \D_t \rho + \dvg q &= 0, \label{eq:ep_nd_mass}\\
  \D_t q + \dvg \left( \frac{q \otimes q}{\rho} \right) + 
  \nabla p(\rho) &= \rho \nabla \phi, \label{eq:ep_nd_mom}\\
  \veps^2 \Delta \phi &= \rho - 1. \label{eq:ep_nd_poi} 
\end{align} 
Here, the independent variables $t\geq 0$ and $x\in\mbb{R}^d$ are,
respectively, the time and space, and the dependent variables
$\rho(t,x)\geq 0, q(t,x) =\rho u, u(t,x) \in\mbb{R}^d$ and
$\phi(t,x)\in\mbb{R}$ are, respectively, the electron density,
electron momentum and electric potential. The parameter  $\veps :=
\lambda_D / x_{\rf}$ is a scaled Debye length; cf.\
\cite{Che13,KT86}. The pressure law in the momentum equation
\eqref{eq:ep_nd_mom} is taken  to be isentropic, i.e.\
$p(\rho)=\rho^\gamma$, with $\gamma>1$. Finally, the spatial operators
$\grd, \dvg$ and $\Delta$ are respectively, the gradient, divergence
and the Laplacian.  For a complete derivation of the above system; see
\cite{Deg13}. 

The one-fluid EP model considered here only takes into account the
negatively charged electrons with the non-dimensionalised scaled
negative charge being equal to $-1$. In the Poisson equation
\eqref{eq:ep_nd_poi}, the factor $1$ on the right hand side term
$\rho-1$ signifies that the plasma consists of a uniform ion
background with a constant ion density $1$ in terms of the scaled
dimensionless variables. The microscopic quantity $\veps$ plays the
role of a singular perturbation parameter; see also \cite{DLV08, Fab92} 
for more details. 

\subsection{Quasineutral Limit}
\label{ch:qn_qn_lt}

The rest of this section is devoted to a brief summary of some
analytical results pertaining to the quasineutral limit of the EP system
\eqref{eq:ep_nd_mass}-\eqref{eq:ep_nd_poi}, deemed relevant for the
analysis to be carried out later, along the lines of \cite{Deg13}. As
discussed in the introduction, the multi-scale nature of the EP system
is elucidated through the parameter $\veps$, when there is a disparity
in the scale of the physical Debye length compared to the macroscopic
length scale. The quasineutral limit model is obtained by letting
$\veps \to 0$ in the system
\eqref{eq:ep_nd_mass}-\eqref{eq:ep_nd_poi}. Further, when $\veps \ll
1$, the flow regime can be considered to be close to the quasineutral
limit.       

Let us assume that the dependent variables $\rho, q$ and $\phi$ yield
the following limits as $\veps \to 0$:
\begin{equation}
  \label{ch:qn_limit_all}
  \lim_{\veps \to 0} \rho = \rho_{(0)}, \quad 
  \lim_{\veps \to 0 } q = q_{(0)} = \rho_{(0)} u_{(0)}, \quad
  \lim_{\veps \to 0} \phi  = \phi_{(0)}.
\end{equation}
Passing to the limit $\veps \to 0$ in
\eqref{eq:ep_nd_mass}-\eqref{eq:ep_nd_poi}, cf.\ \cite{Deg13}, we
obtain the ensuing incompressible Euler system as the quasineutral
limit of the Euler-Poisson system,: 
\begin{align}
  \dvg u_{(0)} &= 0, \label{eq:ep_mass_lim} \\
  \D_t u_{(0)} + \nabla\cdot\left(u_{(0)} \otimes u_{(0)} \right)&= \nabla
                                                                   \phi_{(0)}, \label{eq:ep_mom_lim}\\  
  \rho_{(0)} &=1. \label{eq:ep_poi_lim_}
\end{align} 
Based on the above result we put together the following definition,
which also concerns itself with the choice of the appropriate set of
initial conditions for the EP system
\eqref{eq:ep_nd_mass}-\eqref{eq:ep_nd_poi} and, reveals the multi-scale
nature of the solution in accordance with the formal limits considered
in \eqref{ch:qn_limit_all}.  
\begin{definition}\label{defn:ep_wellprep}
  A triple $(\rho, q, \phi)$ of solution to the EP system
  \eqref{eq:ep_nd_mass}-\eqref{eq:ep_nd_poi}  is called well prepared
  if it admits the following form;
  \begin{align}
    \rho = \rho_{(0)} + \veps^2 \rho_{(2)}, & \quad q = q_{(0)} +
                                              \veps^2
                                              q_{(2)} \label{eq:ep_wp_data_ansz}
    \\ 
    \rho_{(0)} =  1, & \quad \dvg q_{(0)} =
                       0. \label{eq:ep_wp_data_constrnt} 
  \end{align}
\end{definition}
In the incompressible Euler limit system
\eqref{eq:ep_mass_lim}-\eqref{eq:ep_poi_lim_}, the hydrodynamic
pressure is identified as the negative of the limiting electric
potential $\phi_{(0)}$. It is a mixed elliptic-hyperbolic system for
the unknowns $u_{(0)}$ and $\phi_{(0)}$. The electric potential
$\phi_{(0)}$ plays the role a Lagrange multiplier for the
incompressibility constraint \eqref{eq:ep_mass_lim}. Taking divergence
of the momentum equation \eqref{eq:ep_mom_lim}, and using the
divergence-free condition \eqref{eq:ep_mass_lim} on the velocity
$u_{(0)}$, see also \cite{Bre_00}; yields the following elliptic equation
for $\phi_{(0)}$:
\begin{equation}\label{ep:lagrange_mul_phi}
  \Delta \phi_{(0)} = \nabla^2 \colon \left( u_{(0)} \otimes u_{(0)} \right).
\end{equation}
The equation \eqref{ep:lagrange_mul_phi} can be used to eliminate the
divergence constraint \eqref{eq:ep_mass_lim}. In fact, the
divergence-free condition \eqref{eq:ep_mass_lim} needs to be satisfied
only at time $t=0$. We summarise a few results from \cite{Deg13}
as follows.  
\begin{proposition}
  The following reformulated incompressible Euler system:
  \begin{align}
    \Delta \phi_{(0)} &= \nabla^2 \colon \left(u_{(0)} \otimes
                        u_{(0)} \right), \label{eq:ep_mass_ref_lim} \\
    \D_t u_{(0)} + \nabla\cdot\left(u_{(0)}\otimes u_{(0)}\right) &= \grd
                                                                    \phi_{(0)}, \label{eq:ep_mom_ref_lim}\\
    \rho_{(0)} &= 1, \label{eq:ep_poi_ref_lim} \\
    \dvg u_{(0)} (0)& = 0 \label{eq:ep_div_ref}
   \end{align}
  is equivalent to the incompressible Euler system
  \eqref{eq:ep_mass_lim}-\eqref{eq:ep_poi_lim_} for smooth solutions; 
  in the sense that the solution of one system satisfies the other
  system and vice versa.  
\end{proposition}
\begin{remark}
  The divergence-free condition \eqref{eq:ep_div_ref} on the initial
  velocity field in the reformulated Euler-Poisson system plays a
  vital role in obtaining the incompressibility condition for the
  limiting Euler system \eqref{eq:ep_mass_lim}-\eqref{eq:ep_poi_lim_}.
\end{remark}
\begin{remark}
  Even though the reformulated system doesn't serve much of a
  purpose in the analysis of solutions of the Euler-Poisson system or
  the incompressible Euler system, it finds its relevance in the 
  analysis of the numerical schemes to be presented in later
  sections. 
\end{remark}
As is shown in the case of the limiting incompressible system
\eqref{eq:ep_mass_ref_lim}-\eqref{eq:ep_div_ref}, similarly,  the
scaled Euler-Poisson sytem \eqref{eq:ep_nd_mass}-\eqref{eq:ep_nd_poi}
can be also shown to have a reformulated counterpart; see \cite{Deg13}
for a formal derivation of the same.
\begin{proposition}
  The reformulated Euler-Poisson system comprises of the mass
  conservation equation \eqref{eq:ep_nd_mass}, the momentum
  conservation equation \eqref{eq:ep_nd_mom} with a reformulated
  poisson equation and two constraints on the initial conditions given
  as bellow: 
  \begin{align}
    \dvg\left((\veps^2 \D^2_t + \rho) \nabla \phi\right) &=
                                                           \nabla^2\colon\left(\frac{q\otimes
                                                           q}{\rho}+pI\right), \label{eq:ep_poi_ref} 
    \\ 
    \veps^2  \Delta \phi &= \rho - 1, \ \mbox{at} \ t = 0, \label{eq:ep_ic_ref} \\ 
    \veps^2 \Delta \D_t \phi &= \dvg q, \ \mbox{at}  \ t = 0 \label{eq:ep_icd_ref}. 
  \end{align}
  The reformulated system is equivalent to the Euler-Poisson system
  \eqref{eq:ep_nd_mass}-\eqref{eq:ep_nd_poi} as long as their solutions are
  smooth.  
\end{proposition}
\begin{remark}
  As in the case of the reformulated incompressible system with
  respect to the divergence condition on the velocity, here also for
  the compressible system the reformulation has to satisfy the Poisson
  equation and its time derivative, only initially. 
\end{remark}

\section{DAE Formulation and RK Schemes}
\label{sec:ch_qn_time_sd}

The challenges faced in the numerical approximation of the EP system
\eqref{eq:ep_nd_mass}-\eqref{eq:ep_nd_poi} are multifold. First, the
non-evolutionary Poisson equation \eqref{eq:ep_nd_poi} has to be
discretised along with the evolutionary hyperbolic conservation equations
\eqref{eq:ep_nd_mass}-\eqref{eq:ep_nd_mom}. Second, in order to
address the stiffness arising in the quasineutral regime, some
implicitness is to be introduced whilst maintaining the maximal
possible explicitness to achieve computational optimality. Third, the
discretisation should allow high order extensions in space and
time. Fourth, and most importantly, the resulting scheme should be AP
in the quasineutral limit.  

In this section, we address the first challenge mentioned above by
drawing an analogy between the EP system
\eqref{eq:ep_nd_mass}-\eqref{eq:ep_nd_poi} and a class of index-1
DAEs. We refer the reader to, e.g.\ \cite{BCP96} for a comprehensive
treatment of the theory and numerics of DAEs. Subsequently, a brief
summary of the so-called indirect approach \cite{HW96} to derive
implicit RK schemes for DAEs is put on display. Drawing inspiration
from this approach, and to tackle the second concern, two different
IMEX-RK techniques for stiff systems of ODEs are succinctly overviewed
in order to design semi-implicit schemes for the EP system, possessing
the required amount of implicitness.  

\subsection{Linear Subsystem and Index-1 DAEs}
\label{subsec:lin_subsys_ind1_DAE}
To elucidate the parallel between the EP system and index-1 DAEs, we
first linearise the EP system \eqref{eq:ep_nd_mass}-\eqref{eq:ep_nd_poi},
following the lines of \cite{DLV08}, about a stationary homogeneous
state $\bar{\rho} =  1, \ \bar{q} = 0, \ \D_x \bar{\phi} = 0$. Let the
perturbations from the considered stationary state be represented by an
infinitesimal $\delta \ll 1$. An asymptotic expansion of the unknowns
around the constant state with respect to $\delta$ is given by;
\begin{equation}
  \rho = 1 + \delta \rho^{\prime} + \mcal{O} (\delta^2), \quad
  q = \delta u^{\prime} + \mcal{O}(\delta^2), \quad
  \phi = \delta \phi^{\prime} + \mcal{O} (\delta^2).
\end{equation}
Considering only the first order terms in $\delta$, and
dropping the primes for the sake of convenience, we obtain the
following linearised EP system: 
\begin{align}
  \D_t \rho + \D_x u &= 0, \label{eq:ep_mass_lin}\\
  \D_t u + c_{s}^2 \D_x \rho &= \D_x \phi, \label{eq:ep_mom_lin}\\
  \veps^2 \D_{xx} \phi &= \rho, \label{eq:ep_poi_lin}
\end{align}
where $c_s = \sqrt{p^{\prime} (1)}$ is the speed of sound. We take the
partial Fourier transform, in space of the equations
\eqref{eq:ep_mass_lin}-\eqref{eq:ep_poi_lin} to obtain the
following system of equations: 
\begin{align}
  \D_t \hat{\rho} + i \xi \hat{u} &= 0, \label{eq:ep_mass_DAE} \\
  \D_t \hat{u} + i \xi c_s^2 \hat{\rho} & = i \xi \hat{\phi},\label{eq:ep_mom_DAE} \\
  -\veps^2 \xi^2 \hat{\phi} &= \hat{\rho}.\label{eq:ep_poi_DAE}
\end{align}
Here, $\hat{\rho}, \hat{u}$ and $\hat{\phi}$ are the Fourier
transformed variables corresponding to $\rho, u$, and $\phi$,
respectively. Furthermore, to comprehend the analogy between the
Fourier transformed linearised EP system
\eqref{eq:ep_mass_DAE}-\eqref{eq:ep_poi_DAE}, the Fourier space
variable $\xi$ is treated as a fixed parameter. 
\begin{proposition}
  The set of equations \eqref{eq:ep_mass_DAE}-\eqref{eq:ep_poi_DAE}
  constitutes an index-1 DAE system for the unknown vector $V =
  (\hat{\rho}, \hat{u}, \hat{\phi})$ with respect to the time
  variable $t$. 
\end{proposition}
\begin{proof}
  To begin the proof, we first make the observation that
  \eqref{eq:ep_poi_DAE} is a purely algebraic equation free of any 
  differential terms. In order to substantiate the claim in the proposition,
  i.e.\ the equations \eqref{eq:ep_mass_DAE}-\eqref{eq:ep_poi_DAE}
  form a DAE system, let us consider the functions $\mbb{F}, F_1, F_2$
  and $F_3$, defined by  
  \begin{align}
    \mbb{F}(t, V, V^{\prime}) &:= \begin{pmatrix}
      F_{1} (t, V, V^{\prime})\\
      F_{2} (t, V, V^{\prime})\\
      F_{3} (t, V, V^{\prime})
    \end{pmatrix} := \begin{pmatrix}
      \frac{ \D \hat{\rho} }{\D t}+ i \xi \hat{u} \\
      \frac{ \D \hat{u}}{\D t} + i \xi c_s^2 \hat{\rho} - i \xi \hat{\phi} \\
      -\veps^2 \xi^2 \hat{\phi} - \hat{\rho}
    \end{pmatrix}.
  \end{align}
Here we have denoted $V^{\prime}= \Dt V$. Evaluating the Jacobian
$\frac{\D\mbb{F}}{\D V^{\prime}}$ yields  
\begin{align}
  \frac{\D\mbb{F}}{\D V^{\prime}} &= \begin{pmatrix}
    1 & 0 & 0\\
    0 & 1 & 0\\
    0 & 0 & 0
  \end{pmatrix},
\end{align}
which is a non-invertible/singular matrix. Therefore, using the
definition of DAEs, cf.\ \cite{BCP96}, we conclude that the equations 
\eqref{eq:ep_mass_DAE}-\eqref{eq:ep_poi_DAE} constitute a system of
DAEs.   

Next, we address the claim that the DAE system
\eqref{eq:ep_mass_DAE}-\eqref{eq:ep_poi_DAE} is of index-1. Towards
this interest, we differentiate \eqref{eq:ep_poi_DAE} with respect to
$t$ and substitute for $\D_t \hat{\rho}$ from \eqref{eq:ep_mass_DAE},
to get a differential equation for $\hat{\phi}$. At the end, the DAE
system \eqref{eq:ep_mass_DAE}-\eqref{eq:ep_poi_DAE} gets transformed
into    
\begin{align}
  \D_t \hat{\rho} + i \xi \hat{u} &= 0, \label{eq:ep_mass_DAE_ode} \\
  \D_t \hat{u} + i \xi c_s^2 \hat{\rho} - i \xi \hat{\phi} & =
                                                             0,\label{eq:ep_mom_DAE_ode}
  \\ 
  -\veps^2 \xi^2 \D_t \hat{\phi}  + i \xi
  \hat{u}&=0.\label{eq:ep_poi_DAE_ode} 
\end{align}
It can easily be verified that the set of equations
\eqref{eq:ep_mass_DAE_ode}-\eqref{eq:ep_poi_DAE_ode} is purely
differential, and not DAEs, by taking the Jacobian with respect to
$V^{\prime}$, as done above. Since only one differentiation was required to
transform the DAEs \eqref{eq:ep_mass_DAE}-\eqref{eq:ep_poi_DAE} into
the ODEs \eqref{eq:ep_mass_DAE_ode}-\eqref{eq:ep_poi_DAE_ode}, we
conclude that the index of the DAEs is one.
\end{proof}

Motivated by the conclusion of the above proposition for the
linearised EP system \eqref{eq:ep_mass_lin}-\eqref{eq:ep_poi_lin}, we
rewrite the non-linear one-fluid EP system
\eqref{eq:ep_nd_mass}-\eqref{eq:ep_nd_poi} as a set of DAEs via
the following definition.
\begin{definition}
  The DAE formulation of the scaled EP system
  \eqref{eq:ep_nd_mass}-\eqref{eq:ep_nd_poi} is given by    
  \begin{align}
    \D_tU +  H(U, \phi) &= 0,  \label{eq:ep_DAE_form_ode}\\
    G(U, \phi) &= 0, \label{eq:ep_DAE_form_alg}
  \end{align} 
  where 
  \begin{equation}\label{eq:ep_DAE_form_H}
    U := \begin{pmatrix}
    \rho \\
    q 
  \end{pmatrix}, \
  H(U, \phi) := F(U) - S(U, \phi), \  G(U, \phi) := \veps^2 \Delta \phi - \rho + 1,
\end{equation}
and
\begin{equation} 
  F(U) := \begin{pmatrix}
    \dvg q \\
    \dvg \left (\frac{q \otimes q}{\rho} \right) + \grd p(\rho) 
  \end{pmatrix}, \ 
  S(U, \phi) := \begin{pmatrix}
    0 \\
    \rho \grd \phi
  \end{pmatrix}.
\end{equation}
\end{definition}
We underline the fact that the method of lines approach is used to write
the system of PDEs \eqref{eq:ep_nd_mass}-\eqref{eq:ep_nd_poi} in a
DAE form given by
\eqref{eq:ep_DAE_form_ode}-\eqref{eq:ep_DAE_form_alg}. Analogous 
treatment of PDEs using the method of lines philosophy to obtain a
system of ODEs is a regular practice in the literature; see, e.g.\ 
\cite{AS20, ARW95, PR05} and the references therein for examples.   
\begin{remark}
  Note that unlike the yield typically obtained using the method of
  lines, i.e.\ a system of ODEs, in this paper, the semblance is drawn with
  a system of DAEs. 
\end{remark}

\subsection{Indirect Approach for RK Methods for DAEs of Index 1}

Having recast the EP system as a set of DAEs in
\eqref{eq:ep_DAE_form_ode}-\eqref{eq:ep_DAE_form_alg}, we now turn
towards the numerical approximation. It has to be noted that even
though implicit RK schemes were primarily designed for stiff systems
of ODEs, these methods can be systematically extended to even broader
classes of equations, such as DAEs. What follows is a brief exposure
to the so-called indirect approach \cite{HW96} to design implicit RK
schemes for index-1 DAEs. Dedicated to this aim, we consider an
arbitrary set DAEs of index 1 in the form    
\begin{align}
  y^{\prime} &= f(y, z), \label{eq:ep_DAE_I1_1}\\
  0 &= g(y,z), \label{eq:ep_DAE_I1_2}
\end{align} 
where $f$ and $g$ are sufficiently smooth functions, and $y$ and $z$ are
vectors of appropriate dimensions with respect to the functions $f$
and $g$.

We now consider an $s$-stage implicit RK method defined by the triple
$(A,c,\omega)$, where $A=(a_{k,l})$ are the coefficients, $c=(c_k)$
are the intermediate times and $\omega=(\omega_{k})$ are the
weights. The indirect approach \cite{HW96} defines an RK scheme for 
\eqref{eq:ep_DAE_I1_1}-\eqref{eq:ep_DAE_I1_2} with intermediate stages  
$(Y^k,Z^k)$ in the following way: 
\begin{align}
  Y^k &= y^n + \Dlt \sum_{l = 1}^s a_{k,l} f(Y^l, Z^l),\ k = 1, 2,
        \ldots, s, \label{eq:ep_Dapp_RK_1} \\
  0 &= g(Y^k, Z^k), \ k = 1, 2, \ldots, s\label{eq:ep_Dapp_RK_2}, 
\end{align}
where the last row of the coefficient matrix $A$ is taken to be equal
to the vector $\omega$; in other words, the implicit RK scheme is
stiffly accurate (SA); see \cite{HW96}. Hence, the numerical solution
$(y^{n+1}, z^{n+1})$ is defined by $y^{n+1}=Y^s$ and 
\begin{equation} \label{eq:ch0_eq_cond_dapp_n+1}
  g(y^{n+1}, z^{n+1}) = g(Y^s, Z^s) = 0.
\end{equation}
Note that the indirect approach enables us to ensure that the manifold  
\begin{equation} \label{eq:ch0_eq_cond_dapp}
  0 = g(y, z)
\end{equation}
is preserved by all the intermediate stages and automatically by the
update stage. For a detailed discussion on the derivation, validity
and stability analysis of the direct approach of implicit RK schemes
applied to index-1 DAEs, interested readers may refer to \cite{HW96}.
\begin{remark}
  The manifold preserving property \eqref{eq:ch0_eq_cond_dapp_n+1} of
  the indirect approach plays a crucial role in obtaining a consistent
  discretisation of non-evolutionary equations, such as the Poisson
  equation in the EP system \eqref{eq:ep_nd_mass}-\eqref{eq:ep_nd_poi}. 
\end{remark}

The indirect approach presented in \cite{HW96}, and briefed above, in 
its core, takes implicit RK schemes for ODEs and derives implicit
schemes for DAEs. Yes indeed an implicit scheme does fair well
mathematically in the sense that it yields the asymptotic properties,
such as stability and consistency which leads to the AP  property. But,
as a drawback, because of the fully implicit nature of the method,
it becomes computationally very expensive when tackling problems,
especially with nonlinearities, such as the EP system
\eqref{eq:ep_DAE_form_ode}-\eqref{eq:ep_DAE_form_alg}. In order 
to reduce the overdose of implicitness offered by fully implicit RK 
schemes via the indirect approach platform
\eqref{eq:ep_Dapp_RK_1}-\eqref{eq:ch0_eq_cond_dapp_n+1}, in the
following, we consider two semi-implicit RK approaches. The ultimate
aim is to develop RK time steppings for index-1 DAEs and in turn for
the EP system \eqref{eq:ep_DAE_form_ode}-\eqref{eq:ep_DAE_form_alg} as
well, with an optimal portion of implicitness.

\subsection{IMEX-RK Schemes}
Semi-implicit or IMEX schemes trade unconditional stability
of fully implicit schemes in order to draw some gains in the
computational efficiency front. IMEX-RK schemes are a staple for the CFD
community; see, e.g.\ \cite{AS20, ARW95, BMY17, Bos07, BFR16, HMR09,
  NBA+14, PR05, Zho96} for IMEX schemes applied to tackle problems
plagued with stiffness. Introduction of an adequate dosage of
implicitness through IMEX schemes provides an additional amount of
stability which is usually required for methods resolving singular
parameters, while maintaining a pardonable level of computational
performance. In this work we consider two types of IMEX-RK schemes,
namely 
\begin{itemize}
\item additively partitioned IMEX-RK schemes \cite{ARS97, PR01};
\item semi-implicit IMEX-RK schemes \cite{BFR16}.
\end{itemize}
The former one, the class of additively partitioned IMEX-RK schemes,
is primarily designed for stiff ODEs in which the identifiable stiff
and non-stiff parts are summed together. To this end, consider the
following system of additively partitioned stiff ODEs: 
\begin{equation}\label{eq:ch0_add_stiff_ode}
  y^\prime = f(t, y) + \frac{1}{\veps} g(t, y),
\end{equation}
where $y = y(t) \in \mbb{R}^n, \ f, g \colon \mbb{R} \times \mbb{R}^n
\to \mbb{R}^n$, and $\veps>0$ is a small parameter known as the
stiffness parameter. Here, $f$ and $g$ are referred to as,
respectively, the non-stiff and stiff parts.
\begin{definition}\label{defn:ch0_IMEX-RK}
  Consider the stiff system \eqref{eq:ch0_add_stiff_ode} in additive
  form. Let $y^n$ be the numerical solution of
  \eqref{eq:ch0_add_stiff_ode} at time $t^n$, and let $\Dlt$ denote a
  fixed time step. An $s$-stage IMEX-RK scheme updates $y^n$ to
  $y^{n+1}$ through the following $s$ intermediate stages trailed by
  the update stage: 
\begin{align}
  Y^k &= y^n + \Dlt \sum_{\ell = 1}^{k-1} \tilde{a}_{k,\ell} f(t^n +
  \tilde{c}_\ell \Dlt, Y^\ell) +  \Dlt \frac{1}{\veps} \sum_{l = 1}^{s}
  a_{k,l} g(t^n +  c_l \Dlt, Y^l), \ k=1,2,\dots,s,\label{eq:ch0_IMEX-RK_def_Yi}\\
  y^{n+1} &= y^n + \Dlt \sum_{k = 1}^{s} \tilde{\omega}_{k} f(t^n +
  \tilde{c}_k \Dlt, Y^k) +  \Dlt \frac{1}{\veps} \sum_{k = 1}^{s}
            \omega_{k} g(t^n +  c_k \Dlt,
            Y^k). \label{eq:ch0_IMEX-RK_def_Ynp} 
\end{align}
The IMEX-RK scheme
\eqref{eq:ch0_IMEX-RK_def_Yi}-\eqref{eq:ch0_IMEX-RK_def_Ynp} is
characterised by the intermediate step sizes 
$\tilde{c} = (\tilde{c}_1, \tilde{c}_2,\ldots,\tilde{c}_s)$ and $c =
(c_1, c_2,\ldots, c_s)$, the RK coefficients  $\tilde{A} =
(\tilde{a}_{i,j})$ and $A = (a_{i,j})$ and the weights
$\tilde{\omega}=(\tilde{\omega}_{1},
\tilde{\omega}_2,\dots,\tilde{\omega}_{s})$ and $\omega =
(\omega_1, \omega_2,\dots, \omega_s)$. The matrices $\tilde{A}$ and
$A$ are $s \times s$ matrices, where $\tilde{a}_{i,j} = 0$ for $j \geq
i$ for the scheme to be explicit in $f$ and we further impose $a_{i,j}
= 0$ for $j>i$, to get a diagonally Implicit RK (DIRK) scheme. The
IMEX-RK schemes are usually represented in a compact form using a
double Butcher tableaux  
\begin{figure}[htbp]
  \centering
  \small
  \begin{tabular}{c|c}
    $\tilde{c}^T$	&$\tilde{A}$\\
    \hline 
			&$\tilde{\omega}^T$
  \end{tabular}
  \quad
  \begin{tabular}{c|c}
    $c^T$	&$A$\\
    \hline 
                &$\omega^T$
  \end{tabular}
  \caption{Double Butcher tableaux of an IMEX-RK scheme.}
  \label{fig:ch0_db-butcher-tableau-IMEXRK}
\end{figure}
\end{definition}
There exist a large section of stiff problems amongst ODEs wherein
identifying an additive partition of the right hand side into stiff
and non-stiff subparts may not be possible. There are even 
systems where it is possible to find an additive splitting of the
stiff terms, however, stiffness can be associated with a nonlinear
term. Consequently, in this scenario the implementation of additive
IMEX-RK schemes may lead to a need for solving highly nonlinear
equations, which in-turn requires the usage of expensive, iterative
Newton solvers.   

In the light of the aforestated predicaments the second type of
schemes that we consider here, i.e.\ the semi-implicit IMEX-RK schemes,
have been recently presented in \cite{BFR16}. The authors have devised
an efficient platform to derive IMEX-RK schemes for problems where a
multiplicative stiff-nonstiff splitting can be recognised. Further,
the stiffness can be associated to some part of the unknown vector,
and rest of it might be non-stiff. In such scenarios, these schemes
have an advantage that either a linear solver can be employed or a
Newton iteration for a problem with reduced non-linearity can be
engaged; see \cite{BQR+19, BRS18, FR17} for some related works. 

Consider a general stiff-ODE in the form
\begin{equation}
  y^\prime = H(y,y), \label{eq:ch0_Hyy}
\end{equation}
where in $H$ the first argument represents the non-stiff part of $y$
and the second argument the stiff-part. Note that the right hand side is
assumed to have a stiff dependence only through the last argument. 
In order to distinguish these,
we use subscripts and write 
\begin{equation}
  H(y,y) = H(y_E, y_I). \label{eq:ch0_HyEyI}
\end{equation}
It has to be noted that both $y_E$ and $y_I$ represent the same
variable $y$ itself. The subscripts are used to only distinguish
between the stiff and non-stiff components.The class
of problems characterised by the system
\eqref{eq:ch0_Hyy}-\eqref{eq:ch0_HyEyI} is said 
to be in a generalised partitioned form. The additively partitioned
stiff-problems considered above also fall under this general class. In
what follows, we iterate a semi-implicit time discretisation technique
to numerically approximate the solution of
\eqref{eq:ch0_Hyy}-\eqref{eq:ch0_HyEyI} as presented in
\cite{BFR16}. The discretisation is done in such a way 
that the variable $y$ without a stiff dependence is treated
explicitly, and the other part implicitly. The formal definition of
the scheme is given below, where we consider only a DIRK variant; the
interested reader is referred to \cite{BFR16} for other variants and
a detailed exposure to the scheme.   
\begin{definition} \label{defn:ch0_SI-IMEX}
  Suppose that the RK coefficients $\tilde{A}=(\tilde{a}_{k,\ell}), \
A=(a_{k,l})$ and the weights $\omega=(\omega_k)$ are given. A
  semi-implicit IMEX-RK (SI-IMEX-RK) scheme for \eqref{eq:ch0_Hyy}
  updates the numerical solution $y^n$ at time $t^n$ to the solution
  $y^{n+1}$ at time $t^{n+1} = t^n + \Dlt$ through the following $s$
  intermediate steps:    
  \begin{align}
    Y^k_E &= y^n + \Dlt\sum_{\ell = 1}^{k-1}
            \tilde{a}_{k,\ell}H(Y^{\ell}_{E},Y^{\ell}_{I}) , \
            \mbox{for} \ k=1,2,\dots,s, \label{eq:ch0_SI-IM-UE}  \\  
    Y^k_I &= y^n + \Dlt\sum_{l = 1}^{k} a_{k,l} H(Y^{l}_{E},
            Y^{l}_{I}) , \ \mbox{for} \
            k=1,2,\dots,s, \label{eq:ch0_SI-IM-UI} 
  \end{align}
  and the final update step:  
  \begin{equation}
    y^{n+1} = y^n + \Dlt \sum_{k=1}^{s}\omega_kH(Y^{k}_{E},
    Y^{k}_{I}). \label{eq:ch0_SI-IM-UNp} 
  \end{equation}
\end{definition}
Here, the vectors $Y_I$ and $Y_E$ correspond to the implicit and
explicit parts of the intermediate solutions, and the corresponding
update formulae are implicit and explicit, respectively. 
\begin{remark} \label{re:remark_SI_SA}
  In order to facilitate the design of IMEX schemes for the EP system
  \eqref{eq:ep_nd_mass}-\eqref{eq:ep_nd_poi}, based on the
  indirect approach discussed above, henceforth, we consider only
  SI-IMEX-RK schemes with stiffly accurate implicit solvers. 
\end{remark}
As a consequence of the above remark, the numerical solution of the
SI-IMEX-RK schemes under consideration in this paper boils down to the
solution of the last implicit intermediate stage $s$, i.e.,
\begin{equation}\label{eq:ep_SI_stiff_acc_fin}
  y^{n+1} = Y^s_I
\end{equation}
The order conditions of SI-IMEX schemes and additively partitioned
IMEX-RK schemes are more or less the same \cite{BFR16}. Moreover the
SI-IMEX schemes are also represented using a double Butcher
tableaux, cf.\ Figure~\ref{fig:ch0_db-butcher-tableau-IMEXRK}. 

\section{Time Semi-discrete Schemes for the EP System}
\label{sec:ch_qn_time_semi-disc}

The goal of this section is to address the third and fourth
challenges, outlined in Section~\ref{sec:ch_qn_time_sd}, faced by
numerical approximations of the EP system in the quasineutral 
limit: achieving both high orders of accuracy and the AP property. To this
end, we develop two different classes of semi-implicit schemes for the
EP system \eqref{eq:ep_nd_mass}-\eqref{eq:ep_nd_poi}. Both the schemes
are based on the indirect approach for DAEs as discussed in the preceding
section. The main difference between the two schemes is that one is
designed on the additively partitioned IMEX-RK platform and the other
on the SI-IMEX-RK platform. The first scheme, despite being linearly
implicit and hence computationally efficient, when subjected to an
asymptotic analysis reveals that it is not consistent with the
quasineutral limit. As a remedy to the paltry performance of the
additively partitioned IMEX-RK schemes in the asymptotic limit, we
develop the second class of schemes based on the SI-IMEX framework. A
formal consistency analysis is then presented in
Section~\ref{sec:anal_TSD} to establish its AP property. 

\subsection{Additively Partitioned IMEX-RK-DAE Schemes}
\label{subsec:add_part_IMEX-RK-DAE}
In order to design RK schemes for the EP system
\eqref{eq:ep_nd_mass}-\eqref{eq:ep_nd_poi}, henceforth we consider
only its DAE formulation
\eqref{eq:ep_DAE_form_ode}-\eqref{eq:ep_DAE_form_alg}. As a first
step, we identify an additive stiff-non-stiff splitting of the
function $H$ defined in \eqref{eq:ep_DAE_form_ode}. Note that the
equation \eqref{eq:ep_DAE_form_H} already gives a natural splitting of 
$H$ as the sum of the nonstiff hyperbolic flux $F$ and the stiff
source term $S$, cf.\ the asymptotic analysis presented in
Section~\ref{ch:qn_qn_lt}. This splitting corresponds to the so-called
classical splitting reported in the literature; see \cite{Deg13,Fab92}
for further details. Making use of the indirect approach from
Section~\ref{sec:ch_qn_time_sd} with the above choice of $F$ and $S$,
we define an additively partitioned IMEX-RK-DAE scheme as follows.  
\begin{definition}\label{defn:IMEX-RK-DAE}
  An $s$-stage additively partitioned IMEX-RK-DAE scheme for the
  EP system \eqref{eq:ep_DAE_form_ode}-\eqref{eq:ep_DAE_form_alg}
  to update the solution $(U^n, \phi^n)$ at time $t^n$ to $(U^{n+1},
  \phi^{n+1})$ at time $t^{n+1}$ uses the following $s$ intermediate
  stages:
  \begin{align} 
    U^k &= U^n - \Dlt \sum_{\ell = 1}^{k-1} \tilde{a}_{k, \ell}
          F(U^{\ell}) + \Dlt \sum_{\ell = 1}^{k} a_{k, \ell} 
          S(U^{\ell}, \phi^{\ell}), \ k=1,2,\dots,s, \label{eq:ep_ad_imex_ieu}\\
    G \big(U^k, \phi^k\big) &=0, \ 
                                 k=1,2,\dots,s, \label{eq:ep_ad_imex_poi} 
  \end{align}
  and the final update,
  \begin{equation}
    \label{eq:ep_IMEXDAE_un+}
    (U^{n+1}, \phi^{n+1}) = (U^s, \phi^s).
  \end{equation}
\end{definition}
\begin{remark}
  Note that the above IMEX-RK schemes are stiffly accurate, and only
  DIRK variants are considered for the sake of maintaining simplicity
  and optimal computational performance. 
\end{remark}
The numerical solution is computed in the following chronology. For
each $k = 1, 2, \dots, s$, first, the density $\rho^k$ is explicitly
updated using the mass conservation update in the system
\eqref{eq:ep_ad_imex_ieu}. Using the updated value of $\rho^k$ thus
obtained, the linear elliptic problem \eqref{eq:ep_ad_imex_poi} is
solved for $\phi^k$. The momentum $q^k$ can then be obtained using the
computed values of $\rho^k$ and $\phi^k$ from
\eqref{eq:ep_ad_imex_ieu}, which is now an explicit update.  

Assuming that the data  $\rho^n, q^n$, and $\phi^n$ at time
$t^n$ for the above scheme
\eqref{eq:ep_ad_imex_ieu}-\eqref{eq:ep_IMEXDAE_un+} are well-prepared
in the sense of Definition~\ref{defn:ep_wellprep}, for each $k = 1, 2,
\ldots, s$, taking the limit $\veps \to 0$ in the Poisson equation
\eqref{eq:ep_ad_imex_poi} we obtain
\begin{equation}
  \rho^{k}_{(0)} = 1, \  \mbox{for} \ k = 1, 2, \ldots, s.  
\end{equation}
Using the above equation in the mass updates for each $k = 1, 2,
\ldots, s$ we obtain 
\begin{equation}
\dvg q_{(0)}^{k-1} = 0, \  \mbox{for} \ k = 1, 2, \ldots, s. 
\end{equation}
Therefore, in the limit $\veps \to 0$, the additively partitioned
IMEX-RK-DAE scheme \eqref{eq:ep_ad_imex_ieu}-\eqref{eq:ep_IMEXDAE_un+}
boils down to the following limiting scheme:
\begin{align}
  \rho_{(0)}^{n+1} &= 1, \label{eq:ep_IMEX-DAE_lim_rho} \\
  u^{n+1}_{(0)} & = u_{(0)}^n - \Dlt \sum_{\ell = 1}^{k-1}
                  \tilde{a}_{k, \ell} \dvg \Big( u_{(0)}^{\ell}
                  \otimes u_{(0)}^{\ell}  \Big)  + \Dlt
                  \sum_{\ell=1}^k a_{k, \ell} \nabla 
            \phi^{\ell}_{(0)}, \label{eq:ep_IMEX-DAE_lim_mom}\\
  \dvg u^k_{(0)} &= 0, \  \mbox{for} \ k = 1, 2, \ldots, s
                   -1. \label{eq:ep_IMEX-DAE_lim_phi}
\end{align}
Note that the above limiting scheme doesn't provide a way for
consistently caring forward the divergence-free condition to the
numerical solution at time $t^{n+1}$. Precisely
speaking, the IMEX-RK-DAE schemes
\eqref{eq:ep_ad_imex_ieu}-\eqref{eq:ep_IMEXDAE_un+} only offer the
divergence-free condition up-till the $(s-1)^{th}$ stage via
\eqref{eq:ep_IMEX-DAE_lim_phi}, and fail to provide it for
$u^s _{(0)}  = u^{n+1}_{(0)}$. A closer look at the limiting scheme
\eqref{eq:ep_IMEX-DAE_lim_rho}-\eqref{eq:ep_IMEX-DAE_lim_phi}
also reveals that, it lacks a way to obtain all the unknowns. Being
fettered by the divergence-free condition on the leading order
velocity, the  IMEX-RK-DAE schemes fail to be consistent with the
incompressible limit system
\eqref{eq:ep_mass_lim}-\eqref{eq:ep_poi_lim_} when $\veps 
\to 0$ and also donot provide valid recursion to update all the
variable at time $t^{n+1}$. Therefore, the IMEX-RK-DAE scheme cannot
be AP.

A similar analysis of a first order version of the additively
partitioned IMEX-RK-DAE scheme
\eqref{eq:ep_ad_imex_ieu}-\eqref{eq:ep_IMEXDAE_un+} can be found in
\cite{Deg13}. In \cite{Deg13}, the first order scheme is termed as the
classical scheme which can be obtained in the present setting by using
the IMEX-RK coefficients given in
Figure~\ref{fig:ep_DIRK111}. Consequently, the IMEX-RK-DAE schemes can
be conceived as natural higher order extensions of the classical scheme.   
\begin{figure}[htbp]
  \small
  \centering
  \begin{tabular}{c|c c}
    0   & 0 &  0 \\
    0   & 1 &  0 \\
    \hline 
        & 1 & 0	
  \end{tabular}
  \hspace{10pt}
  \begin{tabular}{c|c c }
    0  & 0 & 0 \\
    0  & 0 & 1 \\
    \hline 
       & 0 & 1	
  \end{tabular}
  \caption{Double Butcher tableaux of DIRK(1,1,1) scheme.}
  \label{fig:ep_DIRK111}
\end{figure}
In \cite{DLV08}, based on knowledge from plasma physics, it has been
reported that the stability constraints for the classical scheme are
extremely prohibitive in the quasineutral regime $\veps \ll
1$. Therefore, combining this observation and the above presented
proof of inconsistency, one can conclude that the IMEX-RK-DAE schemes 
derived from additively partitioned IMEX-RK schemes fail to qualify to
being asymptotic preserving. 
\begin{remark}
  Note that it is possible to introduce other flux-splittings in the
  additive framework and develop IMEX-RK-DAE schemes possessing the AP
  property: e.g.\ the mass flux $q$ can be taken in $S$ instead of
  $F$. But, as a consequence of such a choice a non-linear elliptic
  problem has to be computationally tackled using iterative Newton
  solvers. In this paper the endeavour is to stay consistent with the
  continuous system which only has a linear elliptic problem. 
\end{remark}

\subsection{SI-IMEX-RK-DAE Schemes}
\label{subsec:SI-IMEX-DAE-TD}
As a first step towards designing an SI-IMEX-RK scheme based on the
indirect approach for the EP system
\eqref{eq:ep_DAE_form_ode}-\eqref{eq:ep_DAE_form_alg}, we specify the
stiff and non-stiff terms. Consequently, it leads to the identification
of terms that are to be treated implicitly or explicitly in $H$ and
$G$. In the following, tracing the lines of  SI-IMEX schemes
presentation in Section~\ref{sec:ch_qn_time_sd}, we subscript
variables with $I$ and $E$, for those which are to be treated 
implicitly and explicitly, respectively. Based on the asymptotic
analysis presented in Section~\ref{ch:qn_sec2}, we recast the DAE form
\eqref{eq:ep_DAE_form_ode}-\eqref{eq:ep_DAE_form_alg} of the EP system
via the stiff and non-stiff conservative variables $U_I := (\rho_I, 
q_I)^T$  and $U_E := (\rho_E, q_E)^T$, and the stiff potential
$\phi_I$ into,
\begin{align}
  \D_t U + \mbb{H} (U_E, U_I, \phi_I) &= 0, \label{eq:ep_DAE_form_ode_SI} \\
  \mbb{G} (U_I, \phi_I) &= 0, \label{eq:ep_DAE_form_alg_SI}
\end{align}
where 
\begin{equation}
  \begin{aligned}
    \mbb{H}(U_E, U_I, \phi_I) &:= \begin{pmatrix}
      \dvg q_I\\
      \dvg \left( \frac{q_E \otimes q_E}{\rho_E} \right) + \grd
      p(\rho_E)  - \rho_E \phi_I   
    \end{pmatrix}, \\
    \mbb{G}(U_I, \phi_I) &:= -\veps^2 \Delta \phi_I - \rho_I + 1.
  \end{aligned}
\end{equation}
Using the indirect approach, we define the SI-IMEX-RK-DAE scheme for
the above DAE formulation of the EP system as follows. 
\begin{definition}\label{defn:SI-IMEX-RK-DAE}
  An $s$-stage SI-IMEX-RK-DAE scheme for the Euler-Poisson system
  \eqref{eq:ep_DAE_form_ode}-\eqref{eq:ep_DAE_form_alg} to update the 
  solution $(U^n, \phi^n)$ at time $t^n$ to $(U^{n+1}, \phi^{n+1})$ at 
  time $t^{n+1}$ uses the following $s$ intermediate steps:
  \begin{align} 
    U_E^k &= U^n - \Dlt \sum_{\ell = 1}^{k-1} \tilde{a}_{k, \ell}
            \mbb{H} \left(U_E^{\ell}, U_I^{\ell},
            \phi_I^{\ell}\right), \ k=1,2,\dots,s, \label{eq:ep_ue}\\ 
    U_I^k &= U^{n} - \Dlt \sum_{\ell = 1}^k a_{k,\ell} \mbb{H}\left(U_E^{\ell}, U_I^{\ell},
            \phi_I^{\ell}\right),  \ k=1,2,\dots,s, \label{eq:ep_SIIMEXDAE_k_ode}\\ 
    \mbb{G}\left(U^k_I, \phi_I^k\right) &=0, \ 
                                          k=1,2,\dots,s, \label{eq:ep_SIIMEXDAE_k_alg} 
  \end{align}
  and the update stage
  \begin{equation}
    \label{eq:ep_SIIMEXDAE_un+}
    U^{n+1} = U_I^s, \quad \phi^{n+1}=\phi_I^s.
  \end{equation}
\end{definition}
Even though the above scheme seems to have more implicit operations to 
be performed than the one based on additively partitioned approach, we
will show that it can be equivalently reformulated so that it has the
need to carry out same number of implicit operations as the scheme
given in Definition~\ref{defn:IMEX-RK-DAE}. 
\begin{remark}
  Note that owing to Remark~\ref{re:remark_SI_SA}, the final update
  \eqref{eq:ep_SIIMEXDAE_un+} of the SI-IMEX-RK-DAE scheme is same as
  the solution of the last implicit step.
\end{remark}
\begin{definition}
  \label{defn:ep_si-imex_ref}
  The reformulation of the $s$-stage SI-IMEX-RK-DAE scheme
  \eqref{eq:ep_ue}-\eqref{eq:ep_SIIMEXDAE_un+} to update the solution
  $(\rho^n, q^n,\phi^n)$ at time $t^n$ to $(\rho^{n+1}, q^{n+1},
  \phi^{n+1})$ at time $t^{n+1}$ uses $s$ intermediate stages and an
  update stage as given below. For $k=1,2,\dots,s$, the fully explicit
  update of the $k^{th}$ stage reads  
  \begin{align}
    \rho_E^k &= \rho^n -\Dlt \sum_{\ell = 1}^{k-1} \tilde{a}_{k, \ell}
               \dvg q_I^{\ell}, \label{eq:ep_ref_rhoe_TSD}\\ 
    q_E^k &= q^n -\Dlt \sum_{\ell = 1}^{k-1} \tilde{a}_{k, \ell}
            \dvg\left( \frac{q_E^{\ell} \otimes
            q_E^{\ell}}{\rho_E^{\ell}} \right)-\Dlt \sum_{\ell = 1}^{k-1}
            \tilde{a}_{k,\ell} \grd p(\rho_E^{\ell}) + \Dlt \sum_{\ell
            = 1}^{k-1} \tilde{a}_{k,\ell} \rho_E^{\ell} \grd
            \phi_I^{\ell}. \label{eq:ep_ref_qe_TSD}  
  \end{align}
  The explicit part of the $k^{th}$ implicit stage is computed as
  \begin{align}
    \hat{\rho}_I^k &= \rho^n -\Dlt \sum_{\ell = 1}^{k-1} a_{k, \ell} \dvg
                     q_I^{\ell}, \label{eq:ep_SIDAE_ref_rho_hat_k}\\ 
    \hat{q}_I^k &= q^n -\Dlt \sum_{\ell = 1}^{k} a_{k,
                  \ell}\nabla\cdot\left(\frac{q_E^{\ell}\otimes
                  q_E^{\ell}}{\rho_E^{\ell}} \right)-\Dlt \sum_{\ell = 1}^{k} a_{k,\ell} 
                  \grd p(\rho_E^{\ell}) + \Dlt \sum_{\ell = 1}^{k-1} a_{k,\ell} \rho_E^{\ell}
                  \grd \phi_I^{\ell}.  \label{eq:ep_SIDAE_ref_q_hat_k}          
  \end{align}
  The implicit solution is obtained by first solving the linear elliptic
  problem:
  \begin{equation}
    \dvg\left(\left(\veps^2 + \Dlt^2a_{k,k}^2\rho_E^k\right) \grd
      \phi^k_I\right) = \hat{\rho}_I^k
    -
    \Dlt
    a_{k,k}\dvg
    \hat{q}_I^k-1, \label{eq:ep_SIDAE_ref_phi_k}
  \end{equation}
  and then performing the explicit updates:
  \begin{align}
    q_I^k & = \hat{q}_I^k + \Dlt a_{k,k} \rho_E^k \grd
            \phi^k_I,  \label{eq:ep_SIDAE__ref_q_k} \\   
    \rho^k_I  &= \hat{\rho}_I^k - \Dlt a_{k,k} \dvg
                q_I^k, \label{eq:ep_SIDAE__ref_rho_k}
  \end{align}
  chronologically. Finally, the numerical solution is given by
  $(\rho^{n+1}, q^{n+1},  \phi^{n+1})=(\rho^{s}_I, q^{s}_I, \phi^{s}_I)$.   
\end{definition}
\begin{proposition}\label{prop:ep_ref_equiv}
  The SI-IMEX-RK-DAE scheme  \eqref {eq:ep_ue}-\eqref{eq:ep_SIIMEXDAE_un+} and the 
  reformulated time semi-discrete scheme defined in
  Definition~\ref{defn:ep_si-imex_ref}, are equivalent.
\end{proposition}
\begin{proof}
  We prove the proposition by using induction on the number of stages
  of the SI-IMEX-RK-DAE scheme. To start with, we consider the first
  stage, i.e.\ when $k = 1$, which corresponds to a trivial update for
  $\rho^1_E$ and $q_E^1$ given by    
  \begin{equation}
    \rho_E^1 = \rho^n, \quad q^1_E = q^n,  \label{eq:ep_tsd_k1_rho_q}
  \end{equation}
  and $\rho^1_I ,q_I^1$ and $\phi_I^1$ are obtained by solving the first
  implicit intermediate step:
  \begin{align}
    \rho^1_I & = \rho^n - \Dlt a_{1,1} \dvg q_I^1, \label{eq:ep_SIDAE_rho_1}\\
    q_I^1 & = q^n - \Dlt a_{1,1} \dvg \left(\frac{q_E^1 \otimes
            q_E^1}{\rho_E^1}\right) - \Dlt a_{1,1} \grd p(\rho_E^1)+
            \Dlt a_{1,1} \rho_E^1 \grd \phi^1_I,  \label{eq:ep_SIDAE_q_1} \\  
    \veps^2 \Delta \phi^1_I & = \rho^1_I - 1. \label{eq:ep_SIDAE_phi_1} 
  \end{align}
Instead of solving the above linearly implicit system for $\rho^1_I,
q_I^1$ and $\phi_I^1$, we use an elliptic reformulation. To this end
we take the divergence of the momentum update \eqref{eq:ep_SIDAE_q_1}
to obtain $\dvg q_I^1$ and, follow it up by substituting the
expression thus obtained in place of $\dvg q_I^1$ in
\eqref{eq:ep_SIDAE_rho_1} to derive, 
\begin{equation}
  \label{eq:ep_SIDAE_rho_1A}
  \rho^1_I = \rho^n - \Dlt a_{1,1} \dvg q^n+\Dlt^2 a_{1,1}^2 \grd^2
  \colon \left(\frac{q_E^1 \otimes q_E^1}{\rho_E^1} + p(\rho_E^1)
    I\right)-\Dlt^2 a_{1,1}^2 \dvg(\rho_E^1 \grd \phi_I^1). 
\end{equation}
Next, we eliminate $\rho^1_I$ between \eqref{eq:ep_SIDAE_phi_1} and
\eqref{eq:ep_SIDAE_rho_1A} to get the following reformulated
elliptic equation for $\phi_I^1$: 
\begin{equation}\label{eq:ep_s1_eref_schm}
  \dvg\left(\left(\veps^2 + \Dlt^2 a^2_{1,1}\rho_E^1\right)\grd \phi^1_I\right)
  = \rho^n - 1 - \Dlt a_{1,1} \dvg q^n + \Dlt^2 a_{1,1}^2 \grd^2
  \colon \left(\frac{q_E^1 \otimes q_E^1}{\rho_E^1} + p(\rho_E^1)
    I\right) . 
\end{equation}
The above elliptic problem is solved to obtain $\phi_I^1$. The values
of $q_I^1$ and $\rho_I^1$ are then computed using the anteriorly
in-hand value of $\phi_I^1$, via the momentum update
\eqref{eq:ep_SIDAE_q_1} and the mass update \eqref{eq:ep_SIDAE_rho_1},
chronologically. Note that these two updates performed following the
prescribed sequence, are explicit evaluations for $q_I^1$ and
$\rho_I^1$ . Now for each $k = 2$ to $s$, we compute $\rho_E^k$ and  
$q_E^k$  explicitly using \eqref{eq:ep_ue}. Subsequently, using the
evaluated values of $\rho_E^k$ and $q_E^k$, the explicit parts
$\hat{\rho}_I^{k}$ and $\hat{q}_I^{k}$, of the implicit step
\eqref{eq:ep_SIIMEXDAE_k_ode} are calculated using
\eqref{eq:ep_SIDAE_ref_rho_hat_k}-\eqref{eq:ep_SIDAE_ref_q_hat_k}. Finally,
the implicit step \eqref{eq:ep_SIIMEXDAE_k_ode} can be written in
terms of $\hat{\rho}_I^k$ and $\hat{q}_I^k$, in the following compact
form:
\begin{align}
  \rho^k_I  &= \hat{\rho}_I^k - \Dlt a_{k,k} \dvg q_I^k, \label{eq:ep_SIDAE_rho_k} \\
  q_I^k & = \hat{q}_I^k + \Dlt a_{k,k} \rho_E^k \grd
          \phi^k_I,  \label{eq:ep_SIDAE_q_k} \\   
  \veps^2 \Delta \phi^k_I & = \rho^k_I - 1. \label{eq:ep_SIDAE_phi_k}  
\end{align}
In order to solve
\eqref{eq:ep_SIDAE_rho_k}-\eqref{eq:ep_SIDAE_phi_k} for $\rho^k_I$,
$q^k_I$ and $\phi^k_I$, we use an
elliptic reformulation, similar to the one performed for $k = 1$,
followed by two explicit evaluations. Eliminating $\rho_I^k$ and
$q_I^k$ amongst \eqref{eq:ep_SIDAE_rho_k}-\eqref{eq:ep_SIDAE_phi_k},
yields the elliptic equation \eqref{eq:ep_SIDAE_ref_phi_k} for $\phi_I^k$. Once
$\phi_I^k$ is known, $q_I^k$ and $\rho_I^k$ are computed using
explicit evaluations of the momentum update \eqref{eq:ep_SIDAE_q_k}
and the mass update \eqref{eq:ep_SIDAE_rho_k}, successively. Ultimately,
the numerical solution $(\rho^{n+1}, q^{n+1}, \phi^{n+1})$ is updated as 
\begin{equation}
  \rho^{n+1} = \rho_I^{s}, \quad q^{n+1} = q_I^{s}, \quad \phi^{n+1}=
  \phi^{s}_I. 
\end{equation}
Therefore a solution to the SI-IMEX-RK-DAE scheme is a solution of
the reformulated scheme in Definition~\ref{defn:ep_si-imex_ref}. The
converse can be proved in an analogous way, similar to \cite{Deg13}. 
\end{proof}
We note that the above proof also serves as a step by step algorithm
to sequentially obtain the numerical solution of the reformulated
SI-IMEX-RK-DAE
\eqref{eq:ep_ref_rhoe_TSD}-\eqref{eq:ep_SIDAE__ref_rho_k} scheme. We
see that, each stage of the SI-IMEX-RK-DAE   
scheme comprises of three sub-steps: first, an explicit evaluation,
second, solution of an elliptic equation, third and last, another explicit
evaluation to update the solution of an implicit step. In other words,
the numerical solution that we compute in this paper is not the
solution of the actual SI-IMEX-RK-DAE scheme
\eqref{eq:ep_ue}-\eqref{eq:ep_SIIMEXDAE_un+}; rather it is the
solution of a reformulated Poisson equation
\eqref{eq:ep_SIDAE_ref_phi_k} coupled with the mass
\eqref{eq:ep_SIDAE__ref_rho_k} and momentum
\eqref{eq:ep_SIDAE__ref_q_k} updates of the original numerical scheme 
\eqref{eq:ep_ue}-\eqref{eq:ep_SIIMEXDAE_un+} which are now explicit in
nature. 
\begin{remark}
  Although the SI-IMEX scheme
  \eqref{eq:ep_ue}-\eqref{eq:ep_SIIMEXDAE_un+} for the non-linear
  EP system \eqref{eq:ep_nd_mass}-\eqref{eq:ep_nd_poi} is formulated
  as a semi-implicit scheme, the reformulation stated in
  Definition~\ref{defn:ep_si-imex_ref} shows that it is only   
  linearly implicit.
\end{remark}
As observed for the additively partitioned IMEX-RK-DAE scheme here also
we see that by making a suitable choice of the RK tableaux; see
Figure~\ref{fig:ep_imex_euler}, the first order AP scheme presented in
\cite{Deg13, DLV08} can be obtained as:
\begin{align}
  \rho^{n+1} &= \rho^{n} - \Dlt \dvg q^{n+1}, \label{eq:ep_euler_mass}\\
  q^{n+1} &= q^n - \Dlt \dvg \left( \frac{q^n \otimes q^n}{\rho^n} +
            p(\rho^n) I\right) + \Dlt \rho^n \grd
            \phi^{n+1}, \label{eq:ep_euler_mom}\\ 
  \veps^2 \Delta \phi^{n+1} &= \rho^{n+1} -1. \label{eq:ep_euler_phi}
\end{align}
That being the case we come to the understanding that the
SI-IMEX-RK-DAE scheme \eqref{eq:ep_ue}-\eqref{eq:ep_SIIMEXDAE_un+}
serves as a natural higher order extension to the first order scheme
\eqref{eq:ep_euler_mass}-\eqref{eq:ep_euler_phi}. 
\begin{figure}[htbp]
  \centering
  \small
  \begin{tabular}{c|c }
    0   & 0	\\
    \hline 
        & 1	
  \end{tabular}
  \hspace{10pt}
  \begin{tabular}{c|c }
    1   & 1    \\
    \hline 
        & 1	
  \end{tabular}.
  \caption{Double Butcher tableau of SI-IMEX-Euler (1,1,1)
    scheme.} 
  \label{fig:ep_imex_euler}
\end{figure}
We refer the reader to \cite{Deg13, DLV08} and the references therein
for a detailed study of the first order scheme; particularly, its
computational cost compared to the so-called classical scheme;
cf. \cite{Deg13, DLV08}. As stated there it is imperative to
understand that the computational time expense typically associated
with any amount of implicitness doesn't affect our scheme adversely.

\section{Analysis of the Time Semi-discrete SI-IMEX-RK-DAE Scheme}  
\label{sec:anal_TSD}
In this section we study the AP property of the time semi-discrete
scheme introduced in Section~\ref{subsec:SI-IMEX-DAE-TD}. It has been
detailed, e.g.\ in \cite{CDV_ASP_05,CDV07,DLV08}, that standard
(classical) discretisation methods applied to the Euler-Poisson system
do not yield the correct quasineutral limit when subjected to 
very small Debye length $\veps$. In order to overcome this shortcoming
of classical discretisation techniques, an AP scheme which is first
order accurate is proposed and analysed in the above references.

We have observed in the preceding section that the AP scheme
presented in the above mentioned references is a first order variant of the
proposed SI-IMEX-RK-DAE scheme. Under the AP framework these schemes
are designed to perform for a wide range of $\veps$: from non-stiff
regimes $\veps=\mcal{O}(1)$ to stiff regimes where $\veps \to
0$. A scheme's attribute of being AP reflects that it performs
uniformly with respect to $\veps$.  The performance referred here
pertains to the unprejudiced stability - asymptotic stability, and
consistency in the limit of the singular parameter - asymptotic
consistency of the method. In what follows we provide a formal proof
for the asymptotic consistency of the time semi-discrete
SI-IMEX-RK-DAE scheme with the quasineutral limit system and also
post a short discussion regarding the asymptotic stability of the
scheme; see also \cite{DLV08} for related discussions. 

\subsection{Asymptotic Consistency}
\begin{theorem}\label{thm:ep_TSD_AP}
  Suppose that the data at time $t^n$ are well-prepared, i.e.\
  $\rho^n$ and $q^n$ admit the decomposition
  \eqref{eq:ep_wp_data_ansz}-\eqref{eq:ep_wp_data_constrnt}. Then, for
  a SA-DIRK scheme, if the intermediate solutions $(\rho_E^k, q^k_E,
  \phi_E^k)$, $(\rho_I^k, q^k_I, \phi_I^k)$ admit the ansatz
  \eqref{eq:ep_wp_data_ansz}, then it follows that 
  \begin{align}
    \lim_{\veps \to 0} \rho^k_E &= \rho^k_{E,(0)}= 1, \quad \lim_{\veps \to
                                  0} \rho^k_I = \rho^k_{I,(0)}= 1,\\ 
    \lim_{\veps \to 0} \dvg q_E^k &= \dvg q^k_{E, (0)}= 0, \quad \lim_{\veps
                                    \to 0} \dvg q_I^k = \dvg q^k_{I,
                                    (0)}=0, 
  \end{align}
  i.e.\ they are well-prepared. As a result, the numerical solution
  $(\rho^{n+1}, q^{n+1}, \phi^{n+1})^T$ after one timestep is
  consistent with the reformulated incompressible limit
  \eqref{eq:ep_div_ref}-\eqref{eq:ep_mass_ref_lim}; in the sense that
  an SI-IMEX-RK-DAE scheme based on an SA-DIRK scheme transforms to an
  $s$-stage scheme for the incompressible Euler system in the
  quasineutral limit $\veps \to 0$. In other words, the SI-IMEX-RK-DAE 
  scheme \eqref{eq:ep_ue}-\eqref{eq:ep_SIIMEXDAE_k_alg} is
  asymptotically consistent.   
\end{theorem} 
\begin{proof}
  We use induction on the number stages ($s$) to prove the
  theorem. Owing to the equivalence between the SI-IMEX-RK-DAE scheme
  \eqref{eq:ep_ue}-\eqref{eq:ep_SIIMEXDAE_un+} and its reformulation
  \eqref{eq:ep_ref_rhoe_TSD}-\eqref{eq:ep_SIDAE__ref_rho_k}
  cf. Proposition~\ref{prop:ep_ref_equiv}, we prove the consistency
  for the reformulated scheme instead.
  
  As a starting step, consider the first stage, i.e.\ for $k = 1$. We
  have from \eqref{eq:ep_ref_rhoe_TSD} and \eqref{eq:ep_ref_qe_TSD}
  \begin{align}
    \rho_{E, (0)}^1&=\lim_{\veps \to 0}\rho_E^1 = \rho^{n}_{(0)}
                     = 1, \\
    \dvg q_{E,  (0)}^1&=\lim_{\veps \to 0} \dvg q_{E}^1 = \dvg
                        q^n_{(0)} = 0.   
  \end{align}
  As a consequence of the above two equations, we obtain the divergence
  condition, $\dvg u_{E,  (0)}^1=0$. Next, we take equation
  \eqref{eq:ep_SIDAE_ref_phi_k} into consideration.  In the limit of 
  $\veps \to 0$, it furnishes the following elliptic equation, for the
  limiting electric potential $\phi_{I,(0)}^1$: 
  \begin{equation}\label{eq:ep_s1_eref_schm_lim}
    \Delta \phi^1_{I, (0)} = \grd^2 \colon \left(u_{E,(0)}^1 \otimes
      u_{E,(0)}^1\right).
  \end{equation}
  We proceed further by letting $\veps\to 0$ in the implicit mass and
  momentum updates, \eqref{eq:ep_SIDAE__ref_rho_k} and
  \eqref{eq:ep_SIDAE__ref_q_k}, respectively, to yield   
  \begin{align}
    \rho^1_{I, (0)} & = 1 - \Dlt a_{1,1} \dvg q_{I,
                      (0)}^1, \label{eq:ep_SIDAE_rho_1_lim}\\ 
    q_{I, (0)}^1 & = u^n_{(0)} - \Dlt a_{1,1} \dvg \left(u_{E, (0)}^1 \otimes
                   u_{E,(0)}^1\right) + \Dlt a_{1,1} \grd \phi^1_{I,
                   (0)}.  \label{eq:ep_SIDAE_q_1_lim} 
  \end{align}
  Taking divergence of the above equation \eqref{eq:ep_SIDAE_q_1_lim}
  and using the value of $\Delta \phi^1_{I, (0)}$ from \eqref
  {eq:ep_s1_eref_schm_lim} in the expression thus obtain, we get the
  following:
  \begin{equation}
    \dvg q_{I, (0)}^1 = 0,
  \end{equation}
  Now, using the above divergence-free expression for the limit momentum
  in the mass update \eqref{eq:ep_SIDAE_rho_1_lim} forthright yields the
  quasineutrality condition:
  \begin{equation}
    \rho_{I,{(0)}}^1 = 1,
  \end{equation}
  for the limiting density. Summarising the above steps, for the first
  stage of the limiting scheme, we gather the following limiting scheme
  for $k =1$.
  
  For the explicit part:   
  \begin{align}
    \rho_{E, (0)}^1 &=  \rho_{(0)}^n =  1, \\
    \dvg q_{E, (0)}^1 &= \dvg q^n_{(0)} = 0.
  \end{align}
  
  For the implicit part:
  \begin{align}
    \rho^1_{I, (0)} & = 1, \label{eq:ep_SIDAE_rho_1_limf}\\ 
    u_{I, (0)}^1 & = u^n_{(0)} - \Dlt a_{1,1} \dvg \left(u_{E, (0)}^1 \otimes
                   u_{E,(0)}^1\right) + \Dlt a_{1,1} \grd \phi^1_{I,
                   (0)},  \label{eq:ep_SIDAE_q_1_limf} \\ 
    \Delta \phi^1_{I, (0)} &= \grd^2 \colon \left(u_{E,(0)}^1 \otimes
                             u_{E,(0)}^1\right), \label{eq:ep_SIDAE_dvg_1_limf}
  \end{align}
  which provisions the divergence free condition, $\dvg q_{I, (0)}^1 =
  \dvg u_{I, (0)}^1 =  0$. Clearly,
  \eqref{eq:ep_SIDAE_rho_1_limf}-\eqref{eq:ep_SIDAE_dvg_1_limf} is a 
  consistent discretisation of the reformulated incompressible Euler
  system \eqref{eq:ep_mass_lim}-\eqref{eq:ep_poi_lim_}. 
  
  For $k = 2$ onwards, proceeding as done in the case of $k = 1$ we
  can obtain 
  
  For the explicit part:   
  \begin{align}
    \rho_{E, (0)}^{k} &=  \rho_{(0)}^n =  1, \label{eq:ep_SIDAE_rE_s-1_elimf}\\
    u_{E, (0)}^{k} & = u^n_{(0)} - \Dlt \sum_{\ell=1}^{k-1}
                     \tilde{a}_{k,\ell} \dvg \left(u_{E, (0)}^{\ell}
                     \otimes u_{E,(0)}^{\ell}\right) + \Dlt
                     \sum_{\ell=1}^{k-1} \tilde{a}_{k,\ell} \grd
                     \phi^{\ell}_{I,
                     (0)},  \label{eq:ep_SIDAE_uE_s-1_elimf}
  \end{align}
  and consequently, $\dvg q_{E, (0)}^{k} = \dvg u_{E,(0)}^{k}=0$.
  
  For the implicit part:
  \begin{align}
    \rho^{k}_{I, (0)} & = 1, \label{eq:ep_SIDAE_rho_s-1_limf}\\ 
    u_{I, (0)}^{k} & = u^n_{(0)} - \Dlt \sum_{\ell=1}^{k} a_{k,\ell} \dvg
                     \left(u_{E, (0)}^{\ell} \otimes
                     u_{E,(0)}^{\ell}\right) + \Dlt \sum_{\ell=1}^k
                     a_{k,\ell} \grd \phi^{\ell}_{I,
                     (0)},  \label{eq:ep_SIDAE_q_s-1_limf} \\
    \Delta \phi^{k}_{I, (0)} &= \grd^2 \colon \left(u_{E,(0)}^{k}
                               \otimes
                               u_{E,(0)}^{k}\right), \label{eq:ep_SIDAE_phi_s-1_limf} 
  \end{align}
  and ergo, the divergence condition:
  \begin{equation}
    \label{eq:ep_SIDAE_dvg_s-1_limf}
    \dvg q_{I, (0)}^k = \dvg u_{I, (0)}^k = 0. 
  \end{equation}
Since the SI-IMEX-RK scheme under consideration is GSA, the numerical
solution at $t^{n+1}$ is same as the solution of the implicit step for
$k  = s$. Therefore we have  
\begin{align}
  \rho^{n+1}_{(0)} & = 1, \label{eq:ep_SIDAE_rho_np1_limf}\\ 
  u_{(0)}^{n+1} & = u^n_{(0)} - \Dlt \sum_{\ell = 1}^{s}a_{s,\ell}
                  \dvg \left(u_{E, (0)}^{\ell} \otimes 
                  u_{E,(0)}^{\ell}\right) + \Dlt \sum_{\ell =
                  1}^{s}a_{s,\ell} \grd \phi^{\ell}_{I,
                  (0)},  \label{eq:ep_SIDAE_q_np1_limf} \\
  \Delta \phi^{n+1}_{(0)} &= \grd^2 \colon \left(u_{E,(0)}^{s}
                            \otimes
                            u_{E,(0)}^{s}\right) \label{eq:ep_SIDAE_phi_ss_limf}, 
\end{align}
and consequently the divergence condition $\dvg q_{I, (0)}^{n+1} = \dvg
u_{I, (0)}^{n+1} = 0$, i.e.\ the numerical solution
$(\rho^{n+1},q^{n+1},\phi^{n+1})$ 
is also well prepared. Hence, we conclude that the limiting scheme is
a consistent approximation of the reformulated incompressible system.  
\end{proof}
\begin{remark}
  Note that even though the data at time level $t^n$ is
  non-wellprepared, still the SI-IMEX-RK scheme will be able to convert
  the following iterates into well-prepared data. For example, if the
  momentum $q^n$ is not well-prepared but the density $\rho^n$ is,
  then owing to the elliptic problem for $\phi^{n+1}$ leads to the
  divergence free condition for $q^{n+1}$ in the limit of $\veps \to
  0$. Similarly, if none of the initial data were well-prepared then it
  takes two iterates i.e. at time level $t^{n+2}$ the data is made
  well-prepared by the SI-IMEX-RK scheme. This is in alignment with
  the remark in \cite{Deg13} regarding the first order AP scheme. 
\end{remark}
\subsection{Asymptotic Stability}
In \cite{DLV08} the authors have carried out a thorough linearised
$L^2$-stability analysis of a semi-implicit time stepping scheme which
corresponds to the first order variant of the SI-IMEX-RK-DAE scheme
given in Figure~\ref{fig:ep_imex_euler}. Through a Fourier analysis
which captures the effects of a centred space discretisation it has
been discovered that the above scheme scheme fails to be uniformly
stable. As a remedy to this some amount of artificial diffusions have
been introduced into the mass and the momentum updates to retain
stability; see also \cite{CDV07}. Though the numerical viscosity added
is minimal, as a consequence, the transformed system corresponds to an
upwind space discretisation. Furthermore, it has been shown the
stability characteristics of the corrected scheme do not degrade with
$\veps$ in the asymptotic range and elsewhere as well; see also
\cite{CDV16} for a similar analysis towards proving the stability of
a first order SI-IMEX-RK-DAE scheme.    

As is concluded in \cite{DLV08} and reiterated in \cite{CDV16}, it is
imperative to add some amount of artificial viscosity to the
hydrodynamic equations, for the numerical scheme to be
stable in the fully discrete setup. Adhering to this we introduce the
necessary viscosity terms to the SI-IMEX-RK-DAE scheme while designing
the space-time fully discrete scheme to be presented in
Section~\ref{sec:ep_space-time_fully-disc}. A detailed stability
analysis for the second order version of the SI-IMEX-RK-DAE scheme is 
quite a tasking job and is a subject matter for some future work.

\section{Space-Time Fully-Discrete Scheme And Its Analysis}
\label{sec:ep_space-time_fully-disc}
This section is devoted to the construction of a space-time fully-discrete 
scheme corresponding to the SI-IMEX-RK-DAE scheme introduced in
Section~\ref{sec:ch_qn_time_sd}, and its asymptotic analysis. To this
end, we avail a similar approach as is employed in \cite{AS20}, where,
a fully-discrete scheme similar in its spirit to the one to be defined
in this paper,  for the low Mach number limit of the compressible
Euler system was obtained using a coalescence of standard Rusanov-type
flux for the explicit fluxes and second-order central differencing for
the implicit derivative terms. A first-order fully discrete schemes
and its asymptotic analysis can be found in \cite{Deg13,
  DLV08}. Subsequently, in this section, we also briefly discuss the
asymptotic preserving properties of the fully discrete scheme, thus to
be defined.  
 
\subsection{Space-Time Fully-Discrete Scheme}
For ease of presentation, we introduce the following notations for
the flux functions and the source term:
\begin{equation}
  \label{eq:ep_flux_source_m}
  F_m(U)=
  \begin{pmatrix}
    F_{\rho,m}(U)\\
    F_{q,m}(U)
  \end{pmatrix}=
  \begin{pmatrix}
    q_m\\
    \frac{q_m}{\rho}q+pe_m
  \end{pmatrix}, \quad
  S_m(\phi)=
  \begin{pmatrix}
    S_{\rho,m}(\phi)\\
    S_{q,m}(\phi)
  \end{pmatrix}=
  \begin{pmatrix}
    0\\
    \phi e_m,
  \end{pmatrix},
\end{equation}
where $e_m$ denotes the $m^{th}$ unit vector in $\mbb{R}^3$, for
$m=1,2,3$.  We rewrite the reformulated time semi-discrete scheme
\eqref{eq:ep_ref_rhoe_TSD}-\eqref{eq:ep_SIDAE__ref_rho_k} in terms of
\eqref{eq:ep_flux_source_m} as follows. For the explicit stages,
\begin{align}
  \rho_E^k &= \rho^n -\Dlt \sum_{\ell=1}^{k-1}\tilde{a}_{k, \ell}
             \sum_{m=1}^3\Dxm F_{\rho, m} (U_I^{\ell}), \\ 
  q_E^k &= q^n -\Dlt \sum_{\ell=1}^{k-1} \tilde{a}_{k, \ell}
          \sum_{m=1}^3 \Dxm F_{q, m} (U_E^{\ell}) + \Dlt
          \sum_{\ell=1}^{k-1} \tilde{a}_{k, \ell} \rho_E^{\ell}
          \sum_{m=1}^3 \Dxm S_{q,m} (\phi_I^{\ell}). 
\end{align}
For the explicit part of the implicit stages,
\begin{align}
  \hat{\rho}_I^k &= \rho^n -\Dlt \sum_{\ell=1}^{k-1} {a}_{k, \ell}
                   \sum_{m=1}^3 \Dxm F_{\rho, m} (U_I^{\ell}), \\ 
  \hat{q}_I^k &= q^n -\Dlt \sum_{\ell=1}^{k} {a}_{k, \ell}
                \sum_{m=1}^3 \Dxm  F_{q, m} (U_E^{\ell}) + \Dlt
                \sum_{\ell=1}^{k-1} {a}_{k, \ell} \rho_E^{\ell}
                \sum_{m=1}^3 \Dxm S_{q,m} (\phi_I^{\ell}). 
\end{align}
Finally, the implicit solution is obtained by solving,
\begin{equation}
  \sum_{m=1}^3 \Dxm\left(\left(\veps^2 + \Dlt^2a_{k,k}^2\rho_E^k
    \right) \Dxm \phi^k_I\right) =\hat{\rho}_I^k-\Dlt
  a_{k,k}\sum_{m=1}^3
  \D_{x_m}\hat{q}_{I,m}^k-1, \label{eq:ep_sk_eref_schm_dfn}
\end{equation}
and following up with the updates,
\begin{align} 
  q_{I,m}^k  &= \hat{q}_{I,m}^k  + \Dlt a_{k, k} \rho_E^{k}
               \sum_{m=1}^3 \Dxm
               \phi_I^{k}, \label{eq:ep_ref_qi_TSD_re}\\  
  \rho_I^k &= \hat{\rho}_I^k -\Dlt a_{k, k} \sum_{m=1}^3 \Dxm
             q_{I,m}^k. \label{eq:ep_ref_rhoi_TSD_re}    
\end{align}
In order to present the fully-discrete scheme, we introduce the
vector $i = (i_1, i_2, i_3)$ where each $i_m$ for $m = 1,2,3 $
represent $x_m$ the space direction. We further introduce the
following finite difference and averaging operators: e.g.\ in the 
$x_m$-direction  
\begin{equation}
  \label{eq:ep_deltax1}
  \delta_{x_m}w_{i}=w_{i+\frac{1}{2}e_m}-w_{i-\frac{1}{2}e_m},
  \quad
  \mu_{x_m}w_{i}=\frac{w_{i+\frac{1}{2}e_m}+w_{i-\frac{1}{2}e_m}}{2},
\end{equation}
for any grid function $w_{i}$. 
\begin{definition} \label{defn:ep_FD_scheme}
  The $k^{th}$ stage of an $s$-stage fully-discrete SI-IMEX-RK-DAE scheme
  for the Euler-Poisson system is defined as follows. The explicit
  solution is obtained by 
  \begin{align}
    \rho_{E, i}^k &= \rho^n_{i} - \sum_{\ell=1}^{k-1}
                      \tilde{a}_{k, \ell} \sum_{m=1}^3 \nu_{m} 
                      \delta_{x_m} \mcal{F}_{\rho,m, i}^{\ell} \\
    q_{E, i}^k &= q^n_i - \sum_{\ell=1}^{k-1} \tilde{a}_{k, \ell}
                    \sum_{m=1}^3 \nu_{m} \delta_{x_m} \mcal{F}_{q, m,
                 i}^{\ell} + \sum_{\ell=1}^{k-1} \tilde{a}_{k,
                   \ell} \sum_{m=1}^3 \nu_{m} \rho_{E, i}^{\ell}
                   \delta_{x_m} \mcal{S}_{q, m, i}^{\ell}.
  \end{align}
  In an analogous manner, the explicit part of the implicit stage is
  computed via,
  \begin{align}
    \hat{\rho}_{I, i}^k &= \rho^n_{i} -\sum_{\ell=1}^{k-1} {a}_{k,
                            \ell}  \sum_{m=1}^3 \nu_{m} \delta_{x_m}
                            \mcal{F}_{\rho,m,i}^{\ell}, \\ 
    \hat{q}_{I, i}^k &= q^n_i - \sum_{\ell=1}^{k}{a}_{k,\ell}
                         \sum_{m=1}^3 \nu_{m}
                         \delta_{x_m} \mcal{F}_{q, m, i}^{\ell}
                         + \sum_{\ell=1}^{k-1} {a}_{k, \ell}
                         \sum_{m=1}^3 \nu_{m} \rho_{E, i}^{\ell}
                       \delta_{x_m} \mcal{S}_{q, m, i}^{\ell}. 
  \end{align}
  Then the implicit solution is obtained by first solving the discrete
  elliptic problem which is a linear system for $\{ \phi^k_{I,i} \}$
  given by,  
  \begin{equation}
   \sum_{m=1}^3  \frac{\delta_{x_m}}{\Delta x_m} \left(\left(\veps^2 +
       \rho_{E, i}^k\right) \frac{\delta_{x_m}}{\Delta x_m} \phi^k_{I, i}\right) =
    \hat{\rho}_{I,i}^k-a_{k,k}\sum_{m=1}^3
    \nu_m\delta_{x_m}\hat{q}_{I,m, i}^k, \label{eq:ep_sk_eref_schm_dfn_FD_phi} 
  \end{equation}
  followed by the explicit evalutaions, 
  \begin{align}
    q_{I, i}^k &= \hat{q}_{I,i}^k + a_{k, k} \sum_{m=1}^3 \nu_{m}
                   \rho_{E, i}^{k}\delta_{x_m} \phi_{I,
                 i}^{k}, \label{eq:ep_sk_eref_schm_dfn_FD_q} \\ 
    \rho_{I, i}^k &= \hat{\rho}^k_{I,i} - a_{k, k} \sum_{m=1}^3
                      \nu_{m} \delta_{x_m}q_{I,m, i}^k + a_{k,k}
                      \sum_{m=1}^3 \nu_{m} \delta_{x_m}
                      \mcal{G}^k_{I,m,i}.  \label{eq:ep_sk_eref_schm_dfn_FD_rho} 
  \end{align}
Here, the repeated index $m$ takes values in $\{1, 2, 3\}$, $\nu_{m}:=
\frac{\Dlt}{\Delta x_m}$ denote the mesh ratios and $\mcal{F}_{\rho,
  m}, \mcal{F}_{q, m}$ and $\mcal{S}_{m}$ are, respectively, the
numerical fluxes and numerical source term used to approximate the
physical fluxes and physical source term $F_{\rho, m}, F_{q, m}$ and
$S_{m}$. Analogous to \cite{AS20}, in this section we have used a
simple Rusanov-type flux for $\mcal{F}_{q, m}$ and central flux for
$\mcal{F}_{\rho, m}$ and $\mcal{S}_m$.  The term $\mcal{G}_{I,m}$
denotes the artificial viscosity added in the mass equation to adhere
to the stability requirements discussed in Section~\ref{sec:anal_TSD}.
The numerical fluxes, source terms and the viscosity terms are defined
as
\begin{equation}\label{eq:ep_q1_flux_FD_ap}
\begin{aligned}
  &\mcal{F}_{\rho, m, i + \frac{1}{2} e_m}^k
  =\frac{1}{2}\left(F_{\rho,m}\left(U_{I, i + e_m}^{k}\right) +
    F_{\rho,m}\left(U_{I, i}^{k}\right)\right),  \\   
  &\mcal{F}_{q, m, i+\frac{1}{2}e_m}^k = \frac{1}{2}
  \left(F_{q,m}\left(U_{E,i+\frac{1}{2} e_m}^{k,+}\right)+F_{q,
                                        m}\left(U_{E, i+\frac{1}{2}
    e_m}^{k,-}\right) \right)-\frac{\alpha_{m, i + \frac{1}{2} e_m}}{2}
  \left(q_{m, E, i+\frac{1}{2}e_m}^{k,+} - q_{m, E, i+\frac{1}{2}
    e_m}^{k,-}\right), \\
 & \mcal{S}^k_{q,m, i+ \frac{1}{2}e_m} = \frac{1}{2} \left( S_{q,m}\left(\phi_{I, i+e_m}^{k}\right) +
                                    S_{q,m}\left(\phi_{I, i}^{k}\right) \right), \
  \mcal{G}^k_{I,m,i+\frac{1}{2}e_m} = \frac{1}{\nu_m} \left(
                                     \rho_{E,i+e_m}^{k} -
                                     \rho_{E,i}^{k}\right).
                                 \end{aligned}
 \end{equation}
Here, for any conservative variable $w$, we have denoted by
$w^\pm_{i+\frac{1}{2} e_m}$, the interpolated states obtained
using the piecewise linear reconstructions. The wave-speeds are 
computed as, e.g.\ in the $x_1$-direction
\begin{equation}
  \label{eq:wave_speed1}
  \alpha_{1, i
    +\frac{1}{2}e_1}:=2\max\left(\left|\frac{q_{1,E,i+\frac{1}{2}e_1}^{k,-}}{\rho_{E,i+\frac{1}{2}e_1}^{k,-}}\right|,
    \left|\frac{q_{1,E,i+\frac{1}{2}e_1}^{k,+}}{\rho_{E,i+\frac{1}{2}e_1}^{k,+}}\right|\right).  
\end{equation}
\end{definition}
\begin{remark}
  Note that we have only introduced additional viscosity into the
  mass updates but not to the momentum equations. This choice is
  in accordance with the Rusanov-type flux used for the explicit
  momenta flux which introduces the required viscosity by default. 
\end{remark}
\begin{remark}
The space discretisation of the time semi-discrete high order
classical scheme \eqref{eq:ep_ad_imex_ieu}- \eqref{eq:ep_IMEXDAE_un+}
is given as 
\begin{align}
  U^{n+1}_{i} &= U^n_{i} - \sum_{m=1}^3 \nu_{m} \delta_{x_m}
                  \mcal{F}_{m, i}^{n} + \sum_{m=1}^2 \nu_{m} \rho_{i}^{n+1} 
  \delta_{x_m} \mcal{S}_{m, i}^{n+1} \label{eq:ep_clas_schm_dfn_FD_ru} \\
  \veps^2 \sum_{m=1}^3 \left(\frac{\delta_{x_m}}{\Delta x_m}\right)^2
  \phi^{n+1}_{i} &=  \rho_{i}^{n+1} - 1, \label{eq:ep_clas_schm_dfn_FD_phi}
\end{align}
Where $\mcal{F}_m $ in \eqref{eq:ep_clas_schm_dfn_FD_ru} is
approximated using Rusanov fluxes similar to
\eqref{eq:ep_q1_flux_FD_ap}.   
\end{remark}

The eigenvalues of the Jacobians of $F_m(U)$ with respect to $U$ can
be obtained as $\ld_{m,1}= \ld_{m,2} =\frac{q_m}{\rho},
\ld_{m,3}=\frac{q_m}{\rho} - \sqrt{p^{\prime}(\rho)}$ and $\ld_{m,4}=
\frac{q_m}{\rho} + \sqrt{p^{\prime}(\rho)}$. Since the fluxes $F_m(U)$
for the classical scheme are entirely approximated by a Rusanov-type
flux, the timestep $\Dlt$ at time $t^n$ is computed by the CFL
condition: 
\begin{equation}
  \label{eq:ep_CFL}
  \Dlt\max_{i}\max_m\left(\frac{\left|\ld_{m,4}\left(U_{i}^n\right)\right|}{\Delta
      x_m}\right)=\nu, 
\end{equation} 
where $\nu<1$ is the given CFL number. For the SI-IMEX-RK-DAE
scheme the mass flux is approximated using simple central
differences. Hence, the eigenvalues of the Jacobians of the part of
the flux which is approximated by Rusanov-type flux can be obtained as
$\ld_{m,1}=0, \ld_{m,2}= \ld_{m,3}=\frac{q_m}{\rho}$ and $\ld_{m,4}= 2
\frac{q_m}{\rho}$ and the CFL condition is given by \eqref{eq:ep_CFL}. 

A detailed stability analysis of the first order AP scheme is carried
in \cite{DLV08}, where it has been shown that the scheme is stable
under a CFL-like condition independent of $\veps$. However, the
classical scheme
\eqref{eq:ep_clas_schm_dfn_FD_ru}-\eqref{eq:ep_clas_schm_dfn_FD_phi},
as it involves the solution of an elliptic problem containing the
stiffness parameter $\veps$, should additionally satisfy the stability
condition \cite{Fab92} given by 
\begin{equation}\label{eq:ep:stab_poisson}
\Dlt \leq \veps.
\end{equation} 

\subsection{Asymptotic Preserving Property}
\label{sec:ep_fully_disc_AP}
Here, we follow the approach of \cite{AS20} to analyse the high order
space-time fully discrete scheme, in-turn demonstrating that the space
discretisation complements the designed time-discretisation
optimally. The analysis of a first-order space-time fully discrete
scheme for the Euler-Poisson system can be found in \cite{CDV07,
  DLV08}. We prove that the space discretisation coupled with the 
time-discretisation is indeed asymptotically consistent with the
quasineutral limit, i.e.\ the scheme boils down to a consistent
discretisation of the incompressible Euler system when $\veps \to
0$. The stability properties of a first order version of the
SI-IMEX-RK-DAE scheme have been studied in \cite{DLV08}. We do 
not intend to provide the stability considerations here. 
\begin{definition}[Well-prepared data]
  A numerical solution $(\rho_{i}, q_{i}, \phi_{i})$ of the
  Euler-Poisson system is called well prepared if
  \begin{align}
    &\rho_{i} = \rho_{(0), i} + \veps^2 \rho_{(2), i}, \quad
    q_{i} = q_{(0), i} + \veps^2 q_{(2),
      i} \label{eq:ep_FD_wp_decomp}\\ 
    \mbox{and,} \quad \quad & \quad \quad \rho_{(0), i} = 1, \quad
                              \quad \quad \quad \quad
                              \hat{\nabla} \cdot  
     q_{(0), i} = 0 \label{eq:ep_FD_wp_cnstrnt}.
  \end{align}
  Here, $\hat{\nabla}$ is the discrete gradient operator introduced by
  the central numerical fluxes, i.e.\ by replacing the derivatives by
  central differences. 
\end{definition}
\begin{theorem}
  Suppose that the data at time $t^n$ are well-prepared, i.e.\
  $\rho^n$ and $q^n$ admit the decompositions in
  \eqref{eq:ep_FD_wp_decomp}. Then for an SA-DIRK scheme, the
  intermediate solutions $ (\rho_I^k, q^k_I, \phi_I^k)$ satisfy 
  \begin{align}
    \lim_{\veps \to 0} \rho^k_{I,i} &=  \rho^k_{I,(0),i} = 1 \\
    \lim_{\veps \to 0} \hat{\nabla} \cdot q_{I, i}^k &= \hat{\nabla}
    \cdot q^k_{I, (0), i}=0. 
  \end{align}
  Moreover, the numerical solution $(\rho^{n+1}_{i},
  q^{n+1}_{i},\phi^{n+1}_{i})$ after one time-step is consistent
  with the reformulated incompressible limit
  \eqref{eq:ep_mass_ref_lim}-\eqref{eq:ep_div_ref}. Therefore, the 
  scheme  transforms to an $s$-stage scheme for the limit system or in
  other words, the fully discrete scheme in
  Definition~\ref{defn:ep_FD_scheme} is asymptotically consistent. 
\end{theorem}  
\begin{proof}
The proof follows from a combination of the proof of the AP property
for the time semi-discrete system i.e.\ Theorem~\ref{thm:ep_TSD_AP}
and the use of mimetic operators as done in \cite{AS20}.
\end{proof}

\section{Numerical Case Studies}
\label{sec:ep_numer_case_stud}
In this section we present the results of numerical case studies
carried out with both the high-order schemes developed in this
paper. For all the test cases we impose a combination of
periodic and homogeneous Dirichlet boundary conditions for the
hydrodynamic variables and the electric potential, respectively. The
results presented here focus to numerically corroborate the following
three properties as claimed throughout the paper:  
\begin{enumerate}[(i)]
\item asymptotic consistency with the quasineutral limit system;
\item asymptotic convergence, i.e.\ second order convergence to the
  quasineutral limit solution;
\item superior performance of the SI-IMEX-RK-DAE scheme over high
  order classical schemes.
\end{enumerate}
To this end we consider three test problems each one demonstrating some
of the above stated properties of the SI-IMEX-RK-DAE scheme. Note that
we use the LSDIRK(2,2,2) \cite{BFR16} variant of the SI-IMEX-RK scheme
and the ARS(2,2,2) \cite{PR01} variant of the additively partitioned
IMEX-RK scheme. Note that both these variants are stiffly accurate and
L-stable and their Butcher tableaux are given in
Figure~\ref{fig:ep_imexARS_LSD}. 
\begin{figure}[htbp]
  \small
  \centering
  \begin{tabular}{c|c c c}
    0		& 0			& 0			& 0	\\
    $\gamma$	& $\gamma$		& 0 			& 0	\\
    $1$	& $\delta$		& $1 - \delta$		& 0	\\
    \hline 
                & $\delta$	& $1 - \delta$	& $0$
  \end{tabular}
  \hspace{10pt}
  \begin{tabular}{c|c c c}
    $0$	& $0$		& 0			& 0		\\
    $\gamma$	& $0$		& $\gamma$ 		& 0		\\
    $1$	& $0$		& $1 - \gamma$		& $\gamma$	\\
    \hline 
        & $0$		& $1 - \gamma$		& $\gamma$
  \end{tabular}
  \hspace{30pt}
  \begin{tabular}{c|c c}
    0		& 0			& 0\\
    $\sigma$	& $\sigma$		& 0 \\
    \hline 
                &$1 - \gamma$	& $\gamma$
  \end{tabular}
  \hspace{10pt}
  \begin{tabular}{c|c c}
    $\gamma$	& $\gamma$ & $0$\\
    $1$	& $1 - \gamma$		& $\gamma$	\\
    \hline 
        & $1 - \gamma$		& $\gamma$
  \end{tabular}
  \caption{Double Butcher tableaux of the IMEX schemes. Left: ARS
    (2,2,2), and Right: LSDIRK(2,2,2). Here, $\gamma=1-\sqrt{2}/2$,
    $\sigma = 1 / 2 \gamma$ and $\delta=1 - \sigma $.}    
  \label{fig:ep_imexARS_LSD}
\end{figure}
\subsection{A Small Perturbation of a Uniform Quasineutral Plasma}
\label{sec:ep_numer_case_stud_P1}
We consider a quasineutral state where the density $\rho$ is
a constant equal to one and the plasma is assumed to move in the
horizontal direction with a uniform constant velocity one; see
\cite{CDV07}. The constant value of density combined with the Poisson
equation implies that the electric potential vanishes, i.e.\ $\phi =
0$. The aim of this case study is to test the AP scheme's ability to
recover a quasineutral background state independent of the mesh-size:
resolving or not resolving the parameter $\veps$. To this end, a small
cosine perturbation of magnitude $\delta = \veps^2$ is added to the
constant horizontal velocity $u_1$. The initial conditions read  
\begin{equation}
  \rho (0, x_1) = 1.0, \ u_1(0, x_1) = 1.0 + \delta \cos 2 \pi x_1, \
  \phi (0, x_1) = 0.0. 
\end{equation}
Note that the perturbation chosen here pertains to a set of
well-prepared initial data cf.\ Definition \ref{defn:ep_wellprep}. The
parameter $\veps$ is chosen as $\veps = 10^{-4}$ and hence the plasma
frequency $\omega=1/\veps= 10^4$. The plasma is considered in the
domain $[0,1]$. Throughout this test case the AP scheme is subjected
to the advective CFL condition~\ref{eq:ep_CFL}.
The CFL number is set to be $0.45$. 

Figure~\ref{fig:Small_pertur_QN_r_m2} shows the density, velocity and
electric potential at time $t = 0.01$ (100 cycles) computed on a mesh
with $\Delta x = 10^{-4}$, resolving the Debye length. Similarly,
Figure~\ref{fig:Small_pertur_QN_r_m1} shows the above mentioned
components at time $t = 0.1$ (1000 cycles) computed on the same
resolved mesh. The time steps for the classical scheme is not only
restricted by the CFL condition given by \eqref{eq:ep_CFL} but also
by the severe stability restriction \eqref{eq:ep:stab_poisson}. It can
be seen that both the classical and the AP schemes perform quite
well to maintain the well-preparedness of the data over time for each
of the components. 
\begin{figure}[htbp]
  \centering
  \includegraphics[height=0.185\textheight]{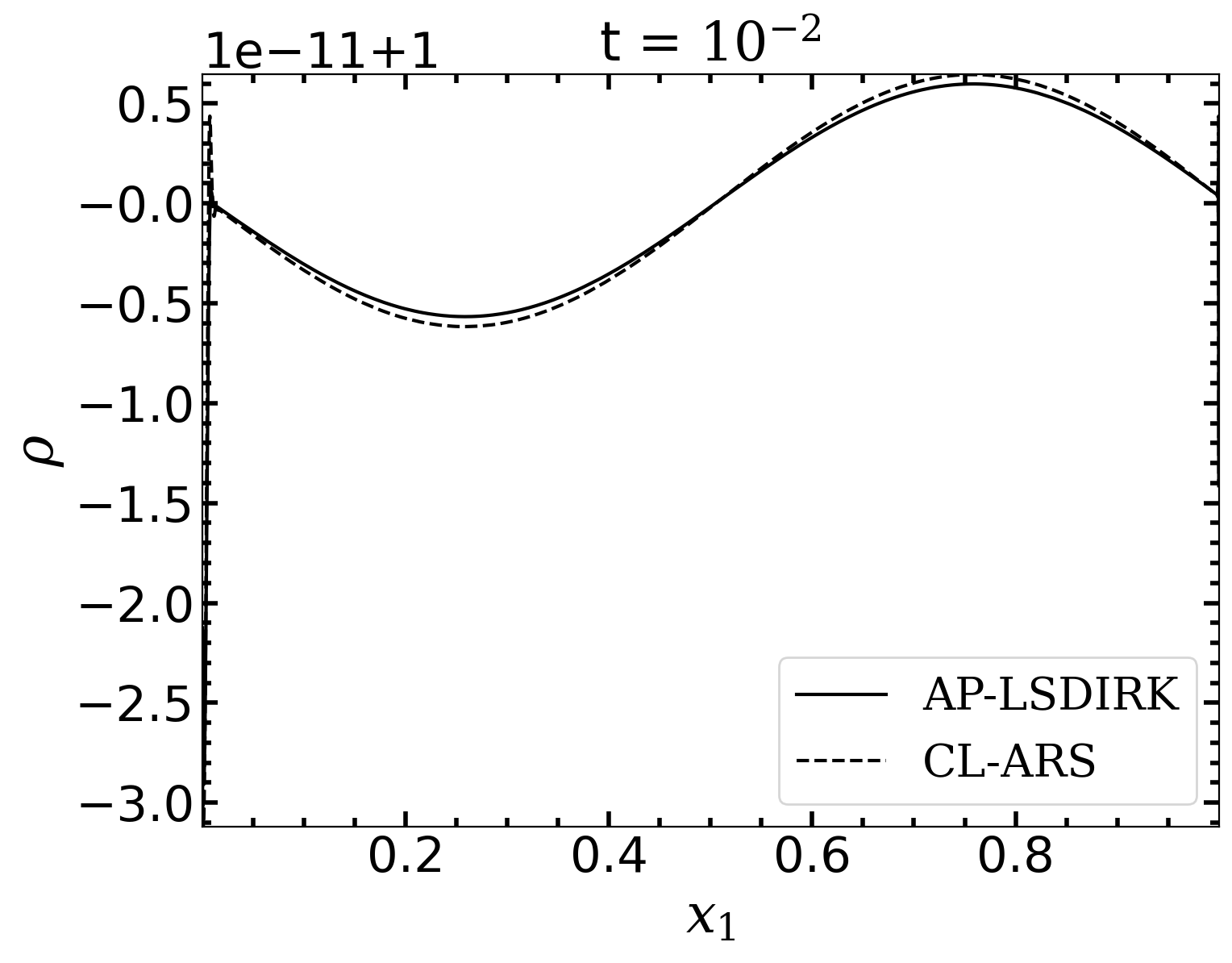}
  \
  \includegraphics[height=0.185\textheight]{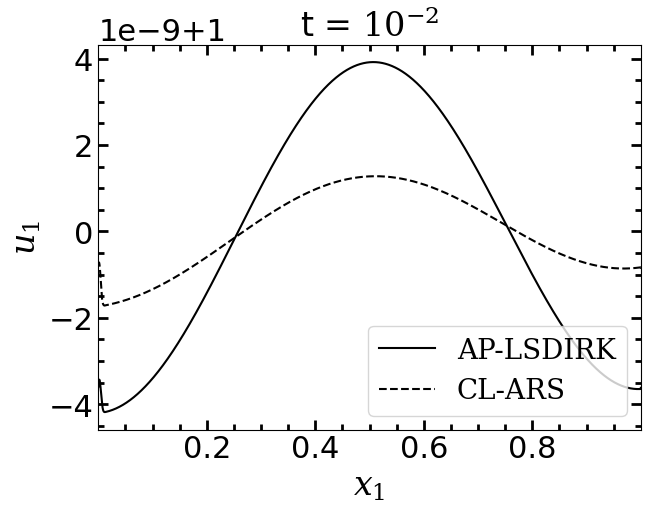}
  \
  \includegraphics[height=0.185\textheight]{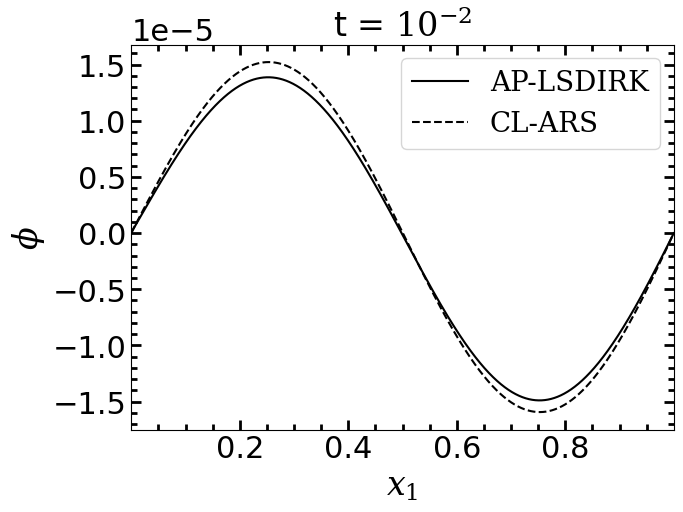}
  \caption{Small perturbation of a quasineutral state with AP and
    Classical Schemes. $\veps = 10^{-4}$,
    $N = 10^4$ (Resolved mesh). Left: Density Profile, Center:
    Velocity Profile, Right: Electric Potential Profiles, at time $t
    = 10^{-2}$.}  
  \label{fig:Small_pertur_QN_r_m2}
\end{figure}
\begin{figure}[htbp]
  \centering
  \includegraphics[height=0.185\textheight]{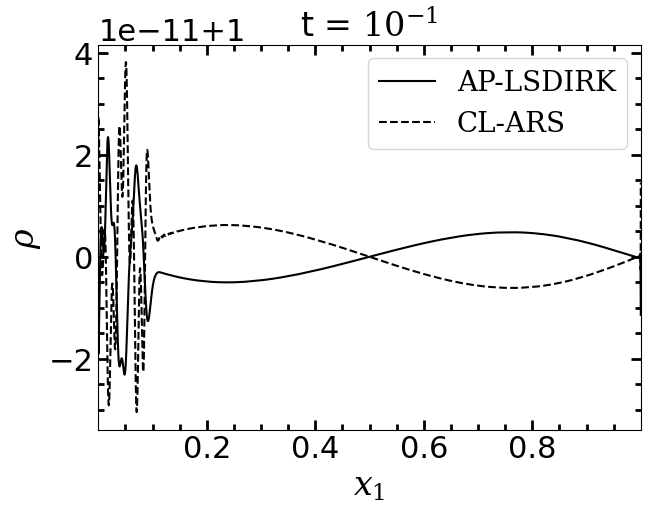}
  \ 
  \includegraphics[height=0.185\textheight]{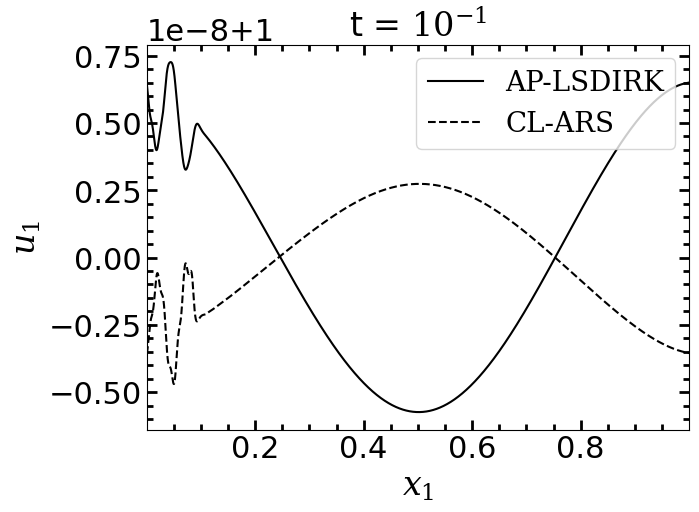}
  \
  \includegraphics[height=0.185\textheight]{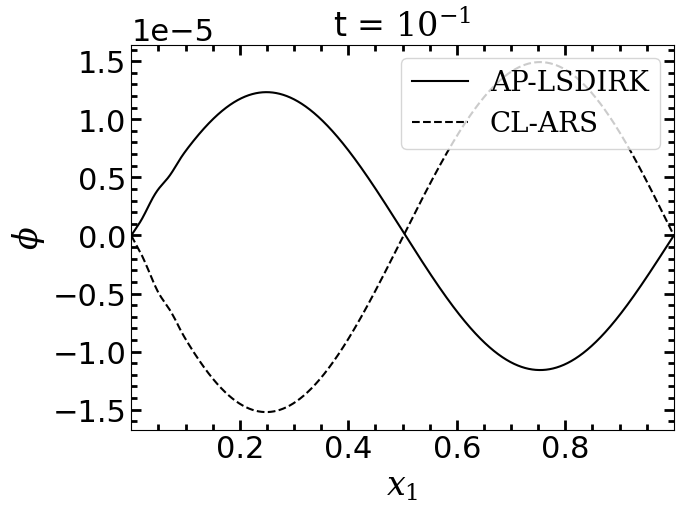}
  \caption{Small perturbation of a quasineutral state with AP and
    Classical Schemes. $\veps = 10^{-4}$,
    $N = 10^4$ (Resolved mesh). Left: Density Profile, Center:
    Velocity Profile, Right: Electric Potential Profiles, at time $t
    = 10^{-1}$.}  
  \label{fig:Small_pertur_QN_r_m1}
\end{figure}

Figure~\ref{fig:Small_pertur_QN_ur_m2} shows the density, velocity and
electric potential at time $t=0.01$ computed on a coarse mesh with
$\Delta x = 10^{-2}$, not resolving the Debye
length. Figure~\ref{fig:Small_pertur_QN_ur_m1} shows the corresponding 
components at time $t = 0.1$ on the same coarse mesh. As was done
in the case of the fine mesh resolving the Debye length here also the
classical scheme is subjected to the CFL conditions in
\eqref{eq:ep_CFL} and \eqref{eq:ep:stab_poisson}. Both the figures
convey that the classical as well as the AP scheme perform well
and retain the well-preparedness aspect of the solution observed with
fine mesh computations. 
\begin{figure}[htbp]
  \centering
  \includegraphics[height=0.185\textheight]{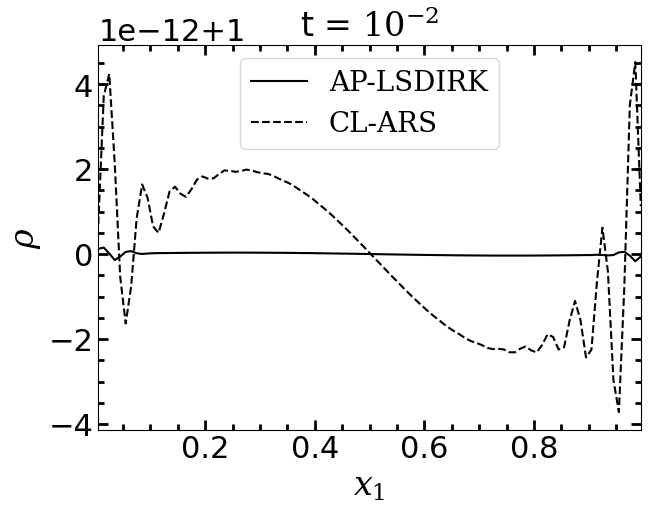}
  \
  \includegraphics[height=0.185\textheight]{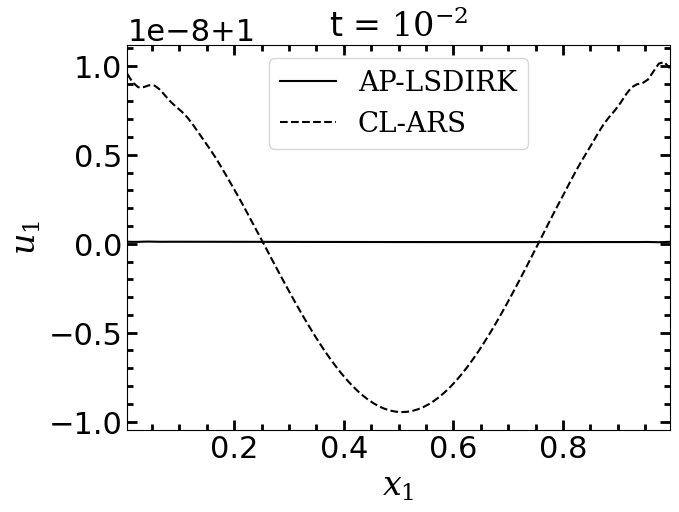}
  \
  \includegraphics[height=0.185\textheight]{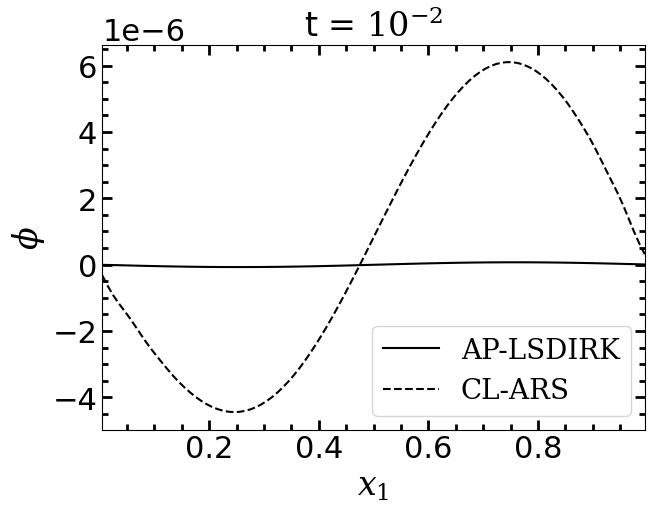}
  \caption{Small perturbation of a quasineutral state with AP ($\Delta
    t > \veps$) and
    Classical Schemes ($\Delta t < \veps$). $\veps = 10^{-4}$,
    $N = 10^2$ (Unresolved mesh). Left: Density Profile, Center:
    Velocity Profile, Right: Electric Potential Profiles, at time $t
    = 10^{-2}$.}  
  \label{fig:Small_pertur_QN_ur_m2}
\end{figure}

\begin{figure}[htbp]
  \centering
  \includegraphics[height=0.185\textheight]{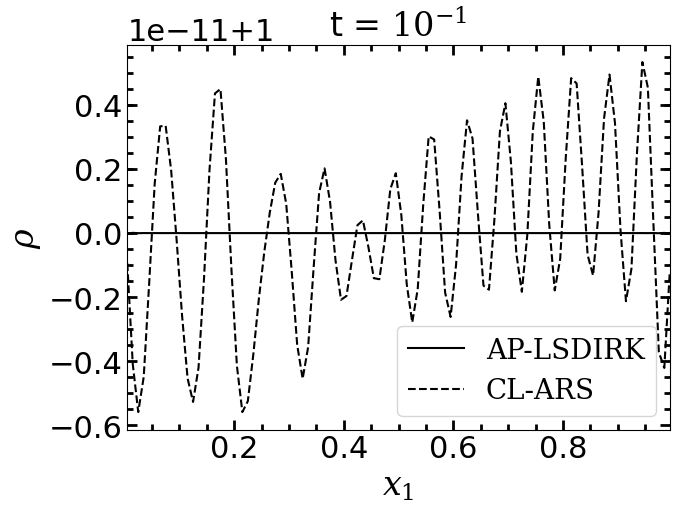}
  \ 
  \includegraphics[height=0.185\textheight]{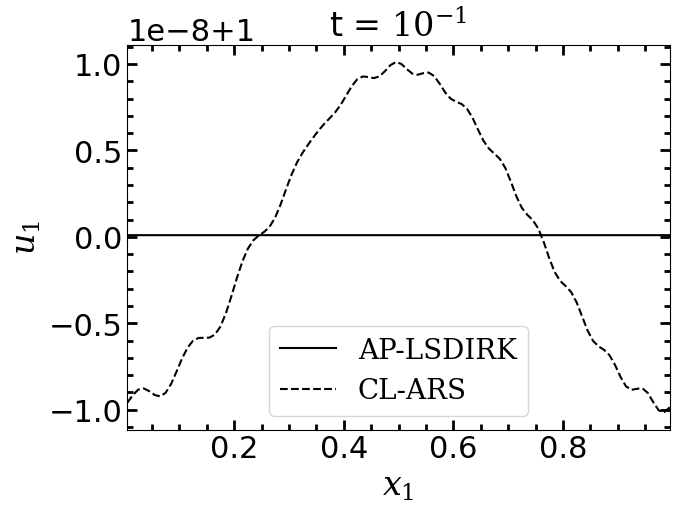}
  \
  \includegraphics[height=0.185\textheight]{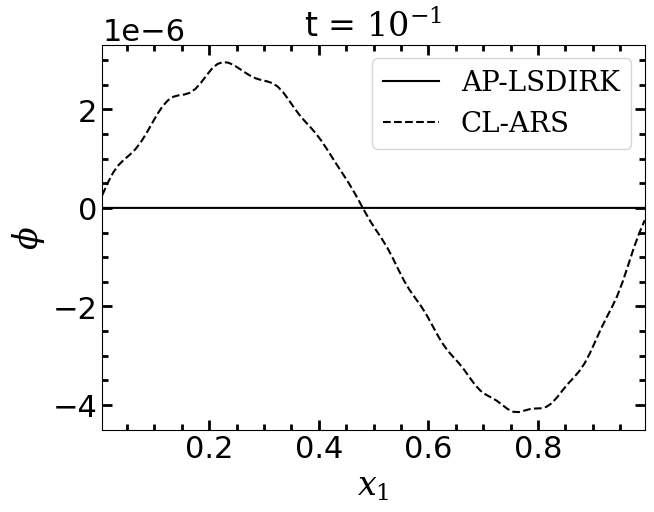}
  \caption{Small perturbation of a quasineutral state with AP ($\Delta
    t > \veps$) and
    Classical Schemes ($\Delta t < \veps$). $\veps = 10^{-4}$,
    $N = 10^2$ (Unresolved mesh). Left: Density Profile, Center:
    Velocity Profile, Right: Electric Potential Profiles, at time $t
    = 10^{-1}$.}  
  \label{fig:Small_pertur_QN_ur_m1}
\end{figure}
\begin{remark}
  Note that in the numerical case studies presented until now the
  classical scheme is restricted to a rather stringent stability
  restriction and its performance which is seen to be similar to that 
  of the AP scheme in case of the coarse mesh can be attributed to
  this strict time-stepping criteria.  
\end{remark}
Owing to the above remark we subject the classical scheme to satisfy
the CFL condition \eqref{eq:ep_CFL} only, and therefore $\Delta t >
\veps$ on a coarse mesh $\Delta x =
10^{-2}$. Figure~\ref{fig:Small_pertur_QN_ur_bu} shows that the  
numerical solution blows up after a very small time $t = 6.1
\times10^{-3} $and hence inferring that the classical scheme is
unstable for $\Delta t > \veps$. Therefore, the AP scheme is able to
capture a solution close to the stationary state with a coarse mesh,
whereas the classical scheme needs a very small time step to be
stable, otherwise it blows up very rapidly. This is in agreement to
the results provided in \cite{CDV07} for the first order classical
scheme for the two-fluid EP system. 
\begin{figure}[htbp]
  \centering
  \includegraphics[height=0.28\textheight]{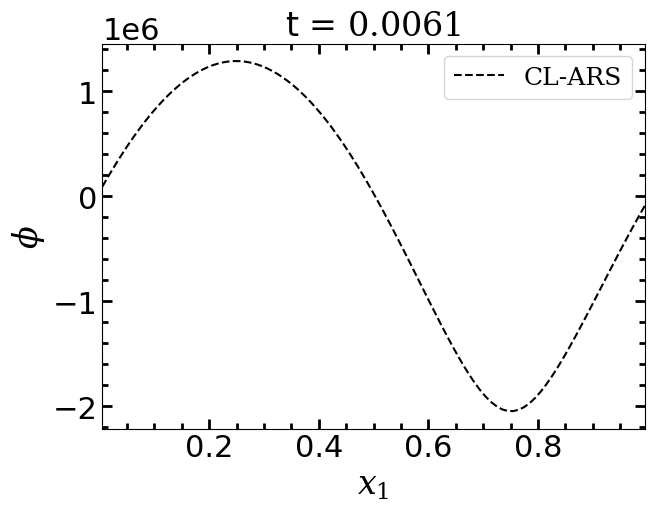} 
\caption{Small perturbation of a  quasineutral state with Classical
  scheme ($\Delta t > \veps$). $\veps = 10^{-4}$, $N = 10^2$
  (unresolved mesh). Electric Potential Profiles, at time $t  =0.0061$}  
  \label{fig:Small_pertur_QN_ur_bu}
\end{figure}

\subsection{Slight Perturbation of a Maxwellian}
Here we consider a non-well prepared initial data. To this end we
consider a test case from \cite{Neg13} where a small perturbation of
amplitude $\delta = 10^{-2}$ to a Maxwellian was considered. The
initial density profile reads:
\begin{equation}
  \rho(0,x_1) = 1.0 + \delta \sin (\kappa  \pi  x_1),
\end{equation}
where the frequency $\kappa = 2220$. The value of $\kappa$ is chosen
such that $\kappa \sim  \veps^{-1}$ to ensure that the wavelength of the
density perturbation is of the same order as the Debye length. This is a
slight perturbation of a stable equilibrium. The expectation from a
numerical scheme is to be able to damp out the oscillations so that
the electric field converges in time towards zero, i.e. the
equilibrium is recovered in the long-time asymptotic.

Here we test both the AP and the classical schemes on a coarse
mesh of 100 mesh points i.e. the mesh size $\Delta x$ not resolving
$\veps = 10^{-4}$ the Debye length. The AP scheme is subjected
only only to the advective CFL condition \eqref{eq:ep_CFL}. But to
adhere to the extra stability requirements of classical schemes, in
addition to \eqref{eq:ep_CFL} the classical scheme is also subjected
to the stability criteria
\eqref{eq:ep:stab_poisson}. Figure~\ref{fig:Small_max_pert_unres} 
shows the density and the electric potential profiles for the AP
scheme and the classical scheme at time $t=0.1$. From this figure we
conclude that the classical scheme even after being put under a far
restrictive time stepping criteria is not able to converge
asymptotically to the equilibrium state on a coarse mesh. Whereas the
AP scheme performs exceedingly well under a very relaxed CFL condition
to obtain a solution close to the quasineutral state. 
\begin{figure}[htbp]
  \centering
  \includegraphics[height=0.283\textheight]{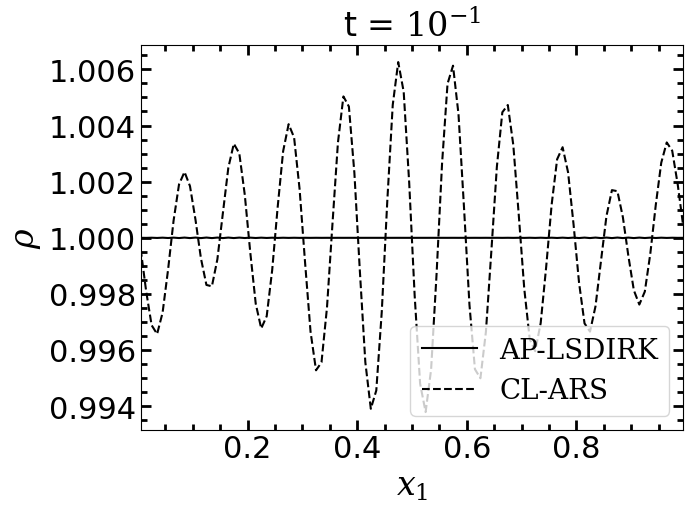}
  \
  \includegraphics[height=0.283\textheight]{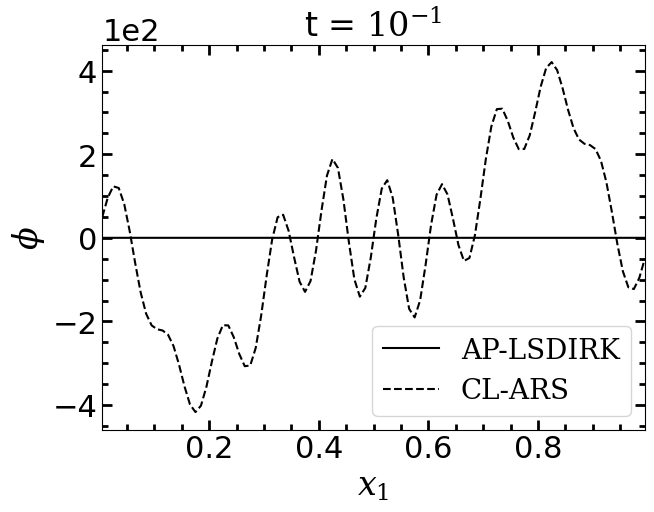}
  \caption{Slight perturbation of a  Maxwellian with AP and Classical
    schemes. $\veps = 10^{-4}$,
    $N = 10^2$ (Unresolved mesh).  Left:Density Profile and
    Right:Electric Potential Profile, at time = 0.1}  
  \label{fig:Small_max_pert_unres}
\end{figure}
\begin{remark}
  Owing to the Remark~1.1 in \cite{Deg13}, we make the
  observation that even if the data is non-well prepared the AP scheme
  is able to damp out the oscillation on a coarse mesh, thereby
  leading towards a robust and a computationally efficient 
  numerical method. 
\end{remark}

\subsection{Asymptotic Order of Convergence}
\label{sec:ep_numer_case_stud_P2}
The goal of this experiment is to numerically validate the second
order asymptotic convergence of the numerical solution to the
incompressible limit solution. Let $v_{\veps}(t^n)$ and $v_{0}(t^n)$
be the solution of the EP system
\eqref{eq:ep_nd_mass}-\eqref{eq:ep_nd_poi} and the incompressible
limit system \eqref{eq:ep_mass_lim}-\eqref{eq:ep_poi_lim_},
respectively, at time $t^n$. Let $v_{\veps}^n$ be the numerical
solution obtained using a second order accurate AP scheme for the EP
system \eqref{eq:ep_nd_mass}-\eqref{eq:ep_nd_poi} at the same
time. Then we expect the following estimates:
\begin{align}
v_{\veps}^{n} - v_{0}(t^n) &= v_{\veps}^{n} -v_{\veps}(t^n)  +
  v_{\veps}(t^n)  - v_{0}(t^n) +  v_{0}(t^n) - v_{0}(t^n) \\
  & = \mcal{O}(\Dlt^2) + \mcal{O}
                                        (\veps) + \mcal{O}(\Dlt^2).
\end{align}
Therefore from the above we get for $\veps \to 0$ we have:
\begin{equation}
  v_{\veps}^{n} - v_{0}(t^n)  = \mcal{O}(\Dlt^2) 
\end{equation}
Consequently, we compute the asymptotic order of convergence, as is
defined in \cite{AS20}, as the experimental rate of convergence of a
numerical scheme with respect to the asymptotic limit
solution. Therefore the exact solution, used as a reference, is the
solution of the incompressible limit
\eqref{eq:ep_mass_lim}-\eqref{eq:ep_poi_lim_}.

To this end we take the initial conditions as considered in
Section~\ref{sec:ep_numer_case_stud_P1} with a modification in the
simulation domain and the value of the perturbation parameter. Here we
expand the plasma domain so that the initial perturbation can travel
within the domain before encountering the boundary conditions at the
domain walls. Consequently, the spatial domain is set to be the
interval $[0, 10]$, and the perturbation parameter is increased to
$\delta = 10^{-2}$.

Here we have consider three different values of $\veps \in \{10^{-4},
10^{-5}, 10^{-6} \}$, all in the asymptotic
range. Table~\ref{table:ep_AOC} displays the second order convergence
of the SI-IMEX-RK-DAE scheme in the quasineutral regime. From
Table~\ref{table:ep_AOC} we observe that even though the error is 
stuck around $\mcal{O}(\veps)$ still the AP scheme is able to recover
the second order convergence. 
\begin{table}
  \begin{tabular}[t] {
    |p{1.3cm}|p{2.2cm}|p{0.8cm}||p{2.2cm}|p{0.8cm}||p{2.2cm}|p{0.8cm}|} 
    \hline
    & $\veps = 10^{-4}$ & & $\veps = 10^{-5}$  & & $\veps = 10^{-6}$
    & \\
    \hline
    $N$& $L^2$ Error in $\phi$ & AOC & $L^2$ Error in
                                                        $\phi$ 
                                               & AOC & $L^2$ Error in $\phi$ & AOC  \\
    \hline
    80  & 1.4641E-01 & & 1.4816E-02 & & 9.7052E-04 & \\
    160 & 1.0083E-02 & 3.85 & 1.0301E-03 & 3.84 & 6.3394E-05 & 3.93\\
    320 & 2.0090E-03 & 2.32 & 2.0869E-04 & 2.30  & 1.3126E-05 & 2.27\\
    640   & 2.6880E-04 & 2.90 & 2.6880E-05 & 2.52 & 2.1804E-06 & 2.58\\
    \hline
  \end{tabular}
  \caption{AOC for SI-IMEX-RK-DAE schemes for $\veps = 10^{-4},
    10^{-5}, 10^{-6}$.}
  \label{table:ep_AOC}
\end{table}

\section{Concluding Remarks}
\label{sec:ch6_con}
In this paper, we have presented a novel procedure which combines two
semi-implicit time discretisation approaches, namely the additively
partitioned IMEX-RK schemes and the SI-IMEX-RK schemes combined with
the direct approach for implicit RK schemes for DAEs, in-order to develop
high-order time semi-discrete schemes for the one-fluid Euler-Poisson
system. The design is leveraged on the fact that via the method of lines, a
linearised version of the governing equations can be perceived as a
system of DAEs. Subsequently, we prove that the time semi-discrete
SI-IMEX-RK-DAE schemes, which are developed on the SI-IMEX-RK platform
are asymptotically consistent with the quasineutral limiting
system. Whereas the schemes based on the additively partitioned IMEX
platform fail to be AP, which is not surprising owing to the analysis of
a first order scheme in \cite{Deg13}. Furthermore, we derive a space-time
fully discrete scheme by approximating the spatial derivatives in a
finite volume framework. Therein we chose a combination of Rusanov-type 
flux and central fluxes for the deemed stiff and the non-stiff terms,
respectively. We show that the space discretisation complements the
time semi-discrete scheme by preserving the AP property at the fully
discrete level. Finally, the results of numerical case studies are provide
to showcase the second order accuracy and the performance of the AP
scheme in the quasineutral and non-quasineutral regimes. The numerical
case studies convey that the AP schemes perform equally well while
resolving or not resolving the Debye length. Whereas the classical
schemes even being put under much stricter stability criterion fail to
replicate the fine mesh results on a coarse mesh. The SI-IMEX-RK
scheme clearly demonstrates the second-order convergence of the
numerical solution to the quasineutral limit solution.

Additionally, we make a note that the schemes developed here are
designed for the EP system, on the account of its simplicity to aid
the analysis and implementation of the same. But, the scheme can be
extended to any stiff system of equations with hyperbolic and elliptic
coupling, in particular for the two-fluid EP system \cite{CDV07}, and
the work is under process and will be reported in the future.

\end{document}